\documentclass[12pt]{amsart}
\usepackage{amssymb,amsmath}
\usepackage{geometry}    
\usepackage{mathrsfs}

\usepackage{vmargin}
\setmargrb{1in}{1in}{1in}{1in}

\newtheorem{theorem}{Theorem}[section]
\newtheorem{lemma}[theorem]{Lemma}
\newtheorem{proposition}[theorem]{Proposition}
\newtheorem{corollary}[theorem]{Corollary}
\theoremstyle{definition}

\newtheorem{remark}[theorem]{Remark}

\begin{document}

\title{Cubic approximation to Sturmian continued fractions}

\author{Johannes Schleischitz} 

\thanks{Research supported by the Schr\"odinger scholarship J 3824 of the Austrian Science Fund FWF.\\
Department of Mathematics and Statistics, University of Ottawa, Canada  \\ 
johannes.schleischitz@univie.ac.at}

\begin{abstract}
We determine the classical exponents of approximation $w_{3}(\zeta),w_{3}^{\ast}(\zeta),\lambda_{3}(\zeta)$
and $\widehat{w}_{3}(\zeta),\widehat{w}_{3}^{\ast}(\zeta),\widehat{\lambda}_{3}(\zeta)$ 
associated to real numbers $\zeta$ whose continued fraction expansions are given by a Sturmian word.
We more generally provide a description
of the combined graph of the parametric successive minima functions defined 
by Schmidt and Summerer  in dimension three for such Sturmian continued fractions. 
This both complements similar results due to Bugeaud and Laurent concerning the two-dimensional exponents 
and generalizes a recent result of the author. As a side result we obtain new information 
on the spectra of the above exponents. Moreover, we provide some information on the exponents
$\lambda_{n}(\zeta)$ for a Sturmian continued fraction $\zeta$ and arbitrary $n$.
\end{abstract}

\maketitle

{\footnotesize{
{\em Keywords}: geometry of numbers, continued fractions, Sturmian words \\
{\em Math Subject Classification 2010}: 11H06, 11J13, 11J70}}

\vspace{4mm}

\section{Exponents of approximation and Sturmian continued fractions} \label{sek1}
We start with the definition of classical exponents of Diophantine approximation. 
Let $\zeta$ be a real transcendental number and $n\geq 1$ be an integer.
For a polynomial $P(T)=a_{n}T^{n}+a_{n-1}T^{n-1}+\cdots+a_{0}$, as usual denote by
$H(P)=\max_{0\leq j\leq n} \vert a_{j}\vert$ its naive height.
Let $w_{n}(\zeta)$ and $\widehat{w}_{n}(\zeta)$ be the supremum of  
$w\in{\mathbb{R}}$ such that the system 
\begin{equation}  \label{eq:w}
H(P) \leq X, \qquad  0<\vert P(\zeta)\vert \leq X^{-w},  
\end{equation}
has a non-zero polynomial solution 
$P(T)\in\mathbb{Z}[T]$ of degree at most $n$
for arbitrarily large $X$, and all large values of $X$, respectively. 
For fixed $\zeta$, the sequences $(w_{n}(\zeta))_{n\geq 1}$ and 
$(\widehat{w}_{n}(\zeta))_{n\geq 1}$
are obviously non-decreasing as $n$ increases.
Dirichlet's Theorem implies the lower bounds
\begin{equation} \label{eq:wmono}
w_{n}(\zeta)\geq \widehat{w}_{n}(\zeta)\geq n.
\end{equation}
Following Bugeaud and Laurent~\cite{buglau}, we define the exponents of 
simultaneous approximation $\lambda_{n}(\zeta)$ and $\widehat{\lambda}_{n}(\zeta)$ as
the supremum of $\lambda\in{\mathbb{R}}$ such that the system 
\begin{equation}  \label{eq:lambda}
1\leq \vert x\vert \leq X, \qquad \max_{1\leq i\leq n} \vert \zeta^{i}x-y_{i}\vert \leq X^{-\lambda},  
\end{equation}
has a solution $(x,y_{1},y_{2},\ldots, y_{n})\in{\mathbb{Z}^{n+1}}$ 
for arbitrarily large values of $X$, and all large $X$, respectively.
In contrast to the polynomial exponents, $(\lambda_{n}(\zeta))_{n\geq 1}$ and $(\widehat{\lambda}_{n}(\zeta))_{n\geq 1}$ 
clearly form non-increasing sequences, as the number of estimates that need to be satisfied in \eqref{eq:lambda}
grows with $n$. 
Another version of Dirichlet's Theorem yields
\begin{equation} \label{eq:ldiri}
\lambda_{n}(\zeta)\geq \widehat{\lambda}_{n}(\zeta)\geq \frac{1}{n},
\end{equation}
 for all $n\geq 1$ and transcendental real $\zeta$.
Khintchine~\cite{khintchine} established the connection
\begin{equation} \label{eq:khintchine}
\frac{w_{n}(\zeta)}{(n-1)w_{n}(\zeta)+n}\leq \lambda_{n}(\zeta)\leq \frac{w_{n}(\zeta)-n+1}{n}.
\end{equation}
Finally we define the quantities $w_{n}^{\ast}(\zeta)$ and $\widehat{w}_{n}^{\ast}(\zeta)$.
For $\alpha$ an algebraic number, we put $H(\alpha)=H(P)$ where
$P$ is its (up to sign) unique
minimal polynomial $P\in\mathbb{Z}[T]$ with coprime
coefficients, and call $H(\alpha)$ the height of $\alpha$.  
Then $w_{n}^{\ast}(\zeta)$ and $\widehat{w}_{n}^{\ast}(\zeta)$ are given as the 
supremum of real $w^{\ast}$ such that 
\[
H(\alpha)\leq X, \qquad 0<\vert \alpha-\zeta\vert\leq H(\alpha)^{-1}X^{-w^{\ast}}
\]
has an algebraic real solution $\alpha$ of degree at most $n$ for arbitrarily large $X$, and all large $X$, respectively. 
For $n=1$, it is not hard to see that $w_{1}(\zeta)=\lambda_{1}(\zeta)=w_{1}^{\ast}(\zeta)$ 
and $\widehat{w}_{1}(\zeta)=\widehat{\lambda}_{1}(\zeta)=\widehat{w}_{1}^{\ast}(\zeta)=1$. 
Clearly again the definition of the exponents directly imply that the sequences
$(w_{n}^{\ast}(\zeta))_{n\geq 1}$ and $(\widehat{w}_{n}^{\ast}(\zeta))_{n\geq 1}$ are non-decreasing.
It is further known that
\begin{equation} \label{eq:sternug}
w_{n}^{\ast}(\zeta)\leq w_{n}(\zeta)\leq w_{n}^{\ast}(\zeta)+n-1, 
\qquad \widehat{w}_{n}^{\ast}(\zeta)\leq \widehat{w}_{n}(\zeta)\leq \widehat{w}_{n}^{\ast}(\zeta)+n-1
\end{equation}
hold for all integers $n\geq 1$ and any transcendental real $\zeta$, 
see~\cite[Lemma~A.8]{bugbuch}. Roughly speaking,
for generic numbers $\zeta$ we expect
equalities in the respective left inequalities of \eqref{eq:sternug}.
Otherwise, in case of strict inequality, for
all polynomials with very small absolute values of evaluations at $\zeta$ (inducing $w_{n}(\zeta)$)
the derivatives at $\zeta$ are untypically small by absolute value. On the other hand, it is known that
$w_{n}(\zeta)-w_{n}^{\ast}(\zeta)$ takes any value in the interval 
$[0,n/2+(n-2)/(4(n-1))]$ for some real 
transcendental $\zeta$, and for $n\in\{2,3\}$ 
any value in the allowed interval $[0,n-1]$, see~\cite{bugduj}.
We will discuss certain spectra for $n=3$
in Section~\ref{spectra}.

We follow the definition of Sturmian words in~\cite[Section~3]{buglau} and enclose
certain properties that are of importance for us. 
Let $(s_{k})_{k\geq 1}$ be a sequence of positive integers and for fixed distinct positive integers $a,b$
consider the words recursively defined by
\begin{equation} \label{eq:m}
m_{0}=b,\quad m_{1}=b^{s_{1}-1}a, \quad m_{k+1}=m_{k}^{s_{k+1}}m_{k-1}, \qquad\qquad k\geq 1.
\end{equation}
Let 
\begin{equation} \label{eq:s}
\varphi=[0;s_{1},s_{2},s_{3},\ldots], \qquad 
m_{\varphi}= \lim_{k\to\infty} m_{k}=b^{s_{1}-1}a\cdots.
\end{equation}
Denote by $\zeta_{\varphi}=[0;m_{\varphi}]$ the number whose continued fraction
expansion is given by concatenation of $0$ and $m_{\varphi}$ and let 
\begin{equation} \label{eq:sigmarder}
\sigma_{\varphi}= \liminf_{k\to\infty} [0;s_{k},s_{k-1},\ldots,s_{1}]=
\frac{1}{\limsup_{k\to\infty}[s_{k};s_{k-1},\ldots,s_{1}]}.
\end{equation}
Throughout the paper let  
\[
\gamma=\frac{1+\sqrt{5}}{2}\approx 1.6180
\]
be the golden ratio. We shall also implicitly consider fixed distinct positive integers $a,b$ 
whenever words $\varphi$ are involved in accordance with the above definitions. As in~\cite{buglau}
denote by $\mathcal{S}$ the set of values $\sigma$ arising in this way,
that is the set of all $\sigma_{\varphi}$ that arises
by all sequences $(s_{k})_{k\geq 1}$ of positive integers via
\eqref{eq:sigmarder}.
It is not hard to see that $\mathcal{S}$ forms a subset of $[0,\gamma^{-1}]$.
The set $\mathcal{S}$ has countable intersection with interval $[s,\gamma^{-1}]$, where $s\approx 0.3867$
is the largest accumulation point of $\mathcal{S}$, but as shown
in~\cite[Theorem~8.2]{buglau} for any $s^{\prime}<s$ the
interval $[s^{\prime},s]$ has uncountable intersection with $\mathcal{S}$. 
On the other hand, as pointed out in the Remark on page~27 in~\cite{buglau}, any interval contained
in $[0,s]$ contains a subinterval which has empty intersection with $\mathcal{S}$.    
For more information on the set $\mathcal{S}$ and the continued fraction expansion
of the value $s$ see~\cite[Section~8]{buglau} and also Cassaigne~\cite{cassaigne}
and Fischler~\cite{fisch1},~\cite{fisch2}. In the special
case of a constant sequence $(s_{k})_{k\geq 1}$ equal to one, the corresponding Sturmian continued fractions satisfy
$\sigma_{\varphi}=\gamma^{-1}=\gamma-1$ and the arising numbers $\zeta_{\varphi}=\zeta_{\gamma-1}$ 
provide examples of extremal numbers defined
by Roy~\cite{royyy}. Extremal numbers attain simultaneously the maximum possible values 
$\widehat{w}_{2}(\zeta)=\widehat{w}_{2}^{\ast}(\zeta)=\gamma+1$ 
and $\widehat{\lambda}_{2}(\zeta)=\gamma-1$ among real $\zeta$ which are not rational or quadratic irrational.
However, Roy~\cite[Theorem~2.4]{droy} proved
the existence of extremal numbers for which no birational equivalent number
has continued fraction expansion given by a Sturmian word (observe that birational
transformations $\zeta\to (a\zeta+b)/(c\zeta+d)$ for $a,b,c,d\in\mathbb{Q}$ with 
$ad-bc\neq 0$ do not affect the classical exponents).   

The classical exponents of Diophantine approximation of the Sturmian continued fractions 
above in dimension $n=2$ have been established by Bugeaud and Laurent~\cite{buglau}.
The central result of their paper is the following~\cite[Theorem~3.1]{buglau}.

\begin{theorem}[Bugeaud, Laurent]  \label{buglauthm}
Let $\zeta=\zeta_{\varphi}$
with corresponding $\sigma=\sigma_{\varphi}$ be as above. Then 
\begin{align*}
\lambda_{2}(\zeta)&= 1, \qquad\qquad w_{2}(\zeta)=w_{2}^{\ast}(\zeta)=1+\frac{2}{\sigma}   \\
\widehat{\lambda}_{2}(\zeta)&=\frac{1+\sigma}{2+\sigma}, 
\qquad \widehat{w}_{2}(\zeta)=\widehat{w}_{2}^{\ast}(\zeta)=2+\sigma.
\end{align*}
\end{theorem} 

The goal of this paper is to determine the exponents for $n=3$. 

\section{The main new results} \label{newr} 

\subsection{Exponents is dimension three}

The main result of the paper is the following.

\begin{theorem} \label{neu}
Let $\zeta=\zeta_{\varphi}$  be a Sturmian continued fraction with corresponding $\sigma=\sigma_{\varphi}$ as above. Then we have
\begin{equation} \label{eq:equ}
w_{3}(\zeta)=w_{3}^{\ast}(\zeta)=1+\frac{2}{\sigma}, \qquad \lambda_{3}(\zeta)= \frac{1}{1+2\sigma},
\end{equation}
and
\begin{equation}  \label{eq:glm}
\widehat{w}_{3}(\zeta)=3, \qquad \widehat{\lambda}_{3}(\zeta)=\frac{1}{3},
\end{equation}
and 
\begin{equation} \label{eq:wstern}
\widehat{w}_{3}^{\ast}(\zeta)= 2+\sigma.
\end{equation}
\end{theorem}

For $\sigma=\gamma^{-1}$ we obtain the results
for extremal numbers of \cite[Theorem~2.1]{ichlondon}, apart from the new claims
for $w_{3}^{\ast}$ and $\widehat{w}_{3}^{\ast}$.
The proof of Theorem~\ref{neu} will lead to the following generalization of~\cite[Theorem~2.2]{ichlondon}. 

\begin{theorem} \label{sternchen}
Let $\zeta=\zeta_{\varphi}$ with corresponding $\sigma=\sigma_{\varphi}$ be as above 
and $\epsilon>0$. Then 
\begin{equation} \label{eq:trio}
\vert Q(\zeta)\vert \leq H(Q)^{-3-\epsilon}
\end{equation}
has only finitely many irreducible solutions $Q\in{\mathbb{Z}[T]}$ of degree precisely three.
In particular
\begin{equation} \label{eq:genaudrei1}
\vert \zeta-\alpha\vert \leq H(\alpha)^{-4-\epsilon} 
\end{equation}
has only finitely many algebraic solutions $\alpha$ of degree precisely three.
On the other hand, the inequalities
\begin{equation} \label{eq:otherhand}
\vert Q(\zeta)\vert \leq H(Q)^{-3+\epsilon}, \qquad \vert \zeta-\alpha\vert \leq H(\alpha)^{-4+\epsilon} 
\end{equation}
have solutions in irreducible polynomials
$P$ and algebraic numbers $\alpha$ of degree precisely three respectively, for 
arbitrarily large $H(Q)$ and $H(\alpha)$. 
Moreover, there exist arbitrarily large values of $X$ such that
\begin{equation} \label{eq:trio2}
H(Q)\leq X, \qquad \vert Q(\zeta)\vert \leq X^{-2\sigma-1-\epsilon}
\end{equation}
has no irreducible solution $Q\in{\mathbb{Z}[T]}$ of degree precisely three.  
Consequently, for arbitrarily large values of $X$ the system
\begin{equation} \label{eq:genaudrei2}
H(\alpha)\leq X, \qquad \vert \zeta-\alpha\vert \leq H(\alpha)^{-1}X^{-2\sigma-1-\epsilon} 
\end{equation}
admits no solution in $\alpha$ an algebraic number of degree precisely three. Conversely, for any large $X$ the estimates
\begin{equation} \label{eq:pere}
\vert Q(\zeta)\vert \leq X^{-2\sigma-1+\epsilon}, \qquad
\vert \zeta-\alpha\vert \leq H(\alpha)^{-1}X^{-2\sigma-1+\epsilon}
\end{equation}
have solutions in irreducible cubic polynomials $Q\in{\mathbb{Z}[T]}$ and cubic
algebraic numbers $\alpha$ of respective height at most $X$.
\end{theorem}   

The estimate \eqref{eq:pere} was only conjectured
for the special case of extremal numbers
in~\cite{ichlondon}, however this estimate as well as \eqref{eq:otherhand} are essentially 
consequences of the method and results of Davenport and 
Schmidt~\cite{davsh}, 
with some minor modifications explained in
Bugeaud~\cite{bugbuch}.
The proof should in fact allow for replacing the $\epsilon$ in the exponents
by suitable multiplicative constants throughout. 
We do not carry this out. 

Finally we discuss the exponents $\lambda_{n}(\zeta_{\varphi})$ for $n>3$.
In contrast to Theorem~\ref{neu} and Theorem~\ref{sternchen}, which are 
generalizations of~\cite[Theorem~2.1 and Theorem~2.2]{ichlondon} with similar proofs, 
no variant of the next theorem was mentioned in~\cite{ichlondon}.
It is essentially derived from specializing 
recent results from~\cite{ichsuccessive}
and~\cite{ichjra}.

\begin{theorem} \label{nvier}
Let $\zeta=\zeta_{\varphi}$ with corresponding $\sigma=\sigma_{\varphi}$
and $n$ be a positive integer. Then 
\begin{equation} \label{eq:az}
\max\left\{\frac{1}{n}\;,\;\frac{2+(3-n)\sigma}{2+(n+1)\sigma}\right\}\leq \lambda_{n}(\zeta)\leq \frac{1}{1+2\sigma}, \qquad\qquad n\geq 3.
\end{equation}
In particular, if and only if $\sigma=0$ we have
\begin{equation} \label{eq:gleidohar}
\lambda_{n}(\zeta)=1, \qquad \qquad n\geq 1.
\end{equation}
On the other hand, if $\sigma>0$, we have the refined upper bounds
\begin{align} 
\lambda_{n}(\zeta)&\leq \min\left\{\frac{1}{(n-1)\sigma}, \frac{1}{1+2\sigma}\right\}, \qquad\qquad
 3\leq n\leq 2+\left\lfloor \frac{2}{\sigma}\right\rfloor,  \label{eq:neddoijesus}  \\
\lambda_{n}(\zeta)&\leq \frac{1}{2+\sigma}<\frac{1}{2}, 
\qquad\qquad\qquad\qquad\qquad n\geq 2+\left\lceil \frac{2}{\sigma}\right\rceil,  \label{eq:neddoije}
\end{align}
and
\begin{equation} \label{eq:ichneu}
\lim_{n\to\infty} \lambda_{n}(\zeta)=0.
\end{equation}
\end{theorem}

For $n=3$ the lower bound in~\eqref{eq:az} becomes the precise value from Theorem~\ref{neu}, but notice there is no equality for $n=2$.
Also there is no reason to believe in equality for any Sturmian continued fraction
and any $n\geq 4$. A better bound for $n=4$ and $\zeta$ an extremal number
was established in~\cite[Theorem~2.3]{ichlondon}. 
In the deduction the identities $w_{2}(\zeta)/\widehat{w}_{2}(\zeta)=\gamma$ and $\gamma^{2}-\gamma-1=0$
played a crucial role.
For general Sturmian continued fractions $\zeta_{\varphi}$ this kind of argument seems no longer to work.
In particular, the proof of~\cite[Theorem~2.3]{ichlondon} suggested equality in the right transference inequality of \eqref{eq:khintchine}
for $\zeta$ an extremal number and $n=4$.
There is no reason for this identity to extend to an arbitrary Sturmian continued fraction $\zeta_{\varphi}$. 
The lower bound in~\eqref{eq:az} is of interest only for $n$ not too large, unless $\sigma=0$. Indeed,
in case of $\sigma>0$, for $n>1+2/\sigma$ we only get
$\lambda_{n}(\zeta)\geq 1/n$ which is the trivial bound from \eqref{eq:ldiri}. 
Notice also that numbers with the property
\eqref{eq:gleidohar} were constructed in~\cite[Theorem~4]{bug} by a similar word concatenation argument.
For $\zeta$ an extremal number, the estimates \eqref{eq:neddoijesus} and \eqref{eq:neddoije} lead to the new information
\[
\lambda_{5}(\zeta)\leq \frac{\sqrt{5}+1}{8}\approx 0.4045, \qquad \lambda_{6}(\zeta)\leq \frac{3-\sqrt{5}}{2}\approx 0.3820.
\]
The bounds are smaller than $\lambda_{3}(\zeta)=1/\sqrt{5}\approx 0.4472$, however
considerably larger than the expected value $(\sqrt{5}-1)/4\approx 0.3090$ for $\lambda_{4}(\zeta)$,
see~\cite[Theorem~2.3]{ichlondon} and the enclosed comments. 
Concerning \eqref{eq:ichneu}, we refer to~\cite[Section~4.3]{ichjra} for a semi-effective minimum rate at which the 
limit $0$ is approached, 
and analogous
results for more general classes of numbers.

\subsection{Spectra}  \label{spectra}

We display immediate consequences of Theorem~\ref{neu} concerning the spectra of 
certain approximation constants. Following~\cite{buglau} we define the spectrum
of an exponent of approximation as the set of real values obtained
as $\zeta$ runs through the real transcendental numbers. 
For convenience we will write $\rm{spec}(.)$, for example $\rm{spec}(\lambda_{3})$.
Recall the properties of the set $\mathcal{S}$ from Section~\ref{sek1}. For a definition 
of Hausdorff dimension see~\cite{falconer}.

\begin{corollary}
For any $\epsilon>0$, the sets
\[
\rm{spec}(\lambda_{3})\cap [1-\epsilon,1], \quad \rm{spec}(\widehat{w}_{3}^{\ast})\cap [2,2+\epsilon] 
,\quad \rm{spec}(\widehat{w}_{3}-\widehat{w}_{3}^{\ast})\cap [1-\epsilon,1]
\]
all have Hausdorff dimension $1$. Moreover
the points $1/(2s+1), 2+s$ and $1-s$ are accumulation points of 
$\rm{spec}(\lambda_{3}), \rm{spec}(\widehat{w}_{3}^{\ast})$ 
and $\rm{spec}(\widehat{w}_{3}-\widehat{w}_{3}^{\ast})$ respectively.
\end{corollary}

\begin{proof}
Theorem~\ref{neu} shows that any $s^{\prime}\in{\mathcal{S}}$ gives rise to a number
$1/(2s^{\prime}+1)$ in $\rm{spec}(\lambda_{3})$, a number $2+s^{\prime}$ in
$\rm{spec}(\widehat{w}_{3}^{\ast})$ and a number $1-s^{\prime}$ in
$\rm{spec}(\widehat{w}_{3}-\widehat{w}_{3}^{\ast})$. On the other hand,
as essentially shown in~\cite{buglau}, 
the intersection of the set $\mathcal{S}$ with $[0,\epsilon)$ has full Hausdorff dimension.
Indeed, the reciprocal of any number with continued fraction expansion
of the form $[K;a_{1},a_{2},\cdots]$ with all $a_{j}<K$, belongs to $\mathcal{S}$
as remarked in~\cite[page~27]{buglau}.
These reciprocal numbers obviously tend to $0$ as $K\to\infty$. On the other hand, 
as remarked in~\cite[page~27]{buglau}, Jarn\'ik showed that 
the dimension of the implied sets tends to $1$ as $K\to\infty$. 
Thus the metrical results follow from the invariance of Hausdorff dimensions under bi-Lipschitz continuous maps.
The latter numbers are accumulation points in the respective spectra since $s$ is an accumulation points of $\mathcal{S}$.
\end{proof}

We expect that $\rm{spec}(\lambda_{3})$ equals the entire interval $[1/3,\infty]$.
This is~\cite[Problem~1]{bug} for $n=3$.
However, only the interval $[1,\infty]$ is known to be included and 
corresponding $\zeta$ can be constructed, as first noticed by Bugeaud~\cite[Theorem~2]{bug}. 
In this interval the numbers $\zeta$ can
even be chosen in the Cantor middle third set, see~\cite[Theorem~4.4]{schlei} and
the more general~\cite[Theorem~2.9]{ichgj}.
Moreover the Hausdorff dimensions of $\{ \zeta\in\mathbb{R}: \lambda_{3}(\zeta)=t\}$
for a prescribed value $t\in[1,\infty]$ are known~\cite[Corollary~1.8]{schlei}. 
On the other hand, very little is known about the remaining interval $[1/3,1)$. In fact the 
first explicit constructions of numbers with prescribed value 
of $\lambda_{n}(\zeta)<1$
for any $n\geq 2$ seem to be extremal numbers $\zeta$ in~\cite{ichlondon}, and $n=3$. 
The spectrum of $\widehat{w}_{3}^{\ast}$ is also supposed to contain the interval $[1,3]$
and $\rm{spec}(\widehat{w}_{3}-\widehat{w}_{3}^{\ast})$ to equal the entire interval $[0,2]$
in which it is contained according to the restriction \eqref{eq:sternug}. However, none 
of these conjectures has been proved. However, as pointed out to me by Y. Bugeaud,
it can be deduced from combination of \cite[Corollary~1]{bug} and~\cite[Theorem~2.2]{buschl} with $m=1$
that $\rm{spec}(\widehat{w}_{n}^{\ast})$ contains $[1,2-1/n]$ 
and $\rm{spec}(\widehat{w}_{n}-\widehat{w}_{n}^{\ast})$
contains $[n-2+1/n,n-1]$, for all $n\geq 2$. The method shows that a sufficient criterion for the 
conjectures above is that numbers $\zeta$ with the property
\[
w_{1}(\zeta)=w_{2}(\zeta)=\cdots=w_{n}(\zeta)=w
\]
exist for all $w\geq n$. This is only known for $w\geq 2n-1$.
Finally we want to investigate the consequences of Theorem~\ref{neu} to $\rm{spec}(w_{3})$.
A metric result of Bernik~~\cite{bernik} implies $\rm{spec}(w_{3})=[3,\infty]$.
However, the numbers $\zeta_{\varphi}$ with $\sigma$ sufficiently close to $\gamma^{-1}=\gamma-1$ 
provide the first explicit constructions of transcendental real numbers with $w_{3}(\zeta)\in(3,5)$. On the other hand 
the subset of $\rm{spec}(w_{3})$ induced by some $\zeta_{\varphi}$ intersected with $(3,5)$ 
consists only of $\{2+\sqrt{5},2+2\sqrt{2}\}$ and is induced by only countably many $\zeta_{\varphi}$, 
as a closer look at the largest few elements of $\mathcal{S}$ shows, 
see again~\cite[page~27]{buglau}. For the remaining interval $[5,\infty]$ the 
construction of Bugeaud~\cite[Corollary~1]{bug} indicated above applies
and yields uncountably many $\zeta$ with prescribed value $w_{3}(\zeta)$. 

We want to point out that Roy~\cite{royexp} proved that $\rm{spec}(\widehat{w}_{2})$ is dense in $[2,\gamma+1]$,
which by Jarn\'iks identity~\cite{jarnik} also shows that $\rm{spec}(\widehat{\lambda}_{2})$ is dense
in $[1/2,\gamma-1]$. He constructed suitable numbers $\zeta$ in a class he called of 
Fibonacci type~\cite[Section~7]{royexp},
which provide a more general concept than Sturmian continued fractions.
A generalization of Theorem~\ref{neu} to Fibonacci type numbers
might lead to density results for
$\rm{spec}(\lambda_{3}), \rm{spec}(\widehat{w}_{3})$ and $\rm{spec}(\widehat{w}_{3}-\widehat{w}_{3}^{\ast})$
in the respective intervals $[1/\sqrt{5},1], [2,1+\gamma]$ and $[2-\gamma,1]$. 
We finally remark that $\rm{spec}(\widehat{w}_{2})$ has countable intersection with
$[c,\gamma+1]$ for some $c<\gamma+1$, in particular it is a proper subset of $[2,\gamma+1]$.

\section{Determination of the uniform exponents}  \label{uni}

\subsection{Proof of \eqref{eq:glm}} \label{glmres}
The main ingredient for the proof of
Theorem~\ref{neu} will be the following Theorem~\ref{vorb}, which is 
stated in a more general form then needed in our applications.
Roughly speaking it claims that if $w_{n-1}(\zeta)=n-1$ and moreover there are infinitely many pairs of
polynomials $P,Q\in\mathbb{Z}[T]$ of degree $n$ whose heights
do not differ too much and both $\vert P(\zeta)\vert$ and $\vert Q(\zeta)\vert$
are very small, then $\widehat{w}_{2n-1}(\zeta)=2n-1$ and $\widehat{\lambda}_{2n-1}(\zeta)=1/(2n-1)$.
This looks somehow exotic at first sight, but if $n=2$ and $\zeta=\zeta_{\varphi}$
the assumptions will be satisfied as can be inferred from results in~\cite{buglau}. 
In particular the claim \eqref{eq:glm} will follow as a direct consequence
of the results in~\cite{buglau} combined with Theorem~\ref{vorb} below, as we will show in this section.
This is also an intermediate step for the proof of \eqref{eq:equ}, however the deduction requires 
additional technical arguments involving parametric
geometry of numbers, very similar to the proof of~\cite[Theorem~2.1]{ichlondon}. 

\begin{theorem} \label{vorb}
Let $n\geq 2$ be an integer and $\zeta$ be a real number.
Assume $w_{n-1}(\zeta)=n-1$. Let $\epsilon>0$ arbitrarily small.
Assume there exist  (not necessarily disjoint) sequences of
polynomials $(P_{i})_{i\geq 1}$ and $(Q_{i})_{i\geq 1}$ 
in $\mathbb{Z}[T]$ of degree at most $n$ with the following properties:
\begin{itemize} 
 \item $\{P_{i},Q_{i}\}$ is linearly independent for $i\geq 1$
\item $H(P_{1})<H(P_{2})<\cdots$ and $H(Q_{1})<H(Q_{2})<\cdots$
\item $H(P_{i})<H(Q_{i})$ for $i\geq 1$ 
\item For large $i\geq i_{0}(\epsilon)$ we have
\begin{equation}  \label{eq:fra} 
\vert P_{i}(\zeta)\vert\leq H(P_{i})^{-2n+1+\epsilon},
 \qquad \vert Q_{i}(\zeta)\vert\leq H(Q_{i})^{-2n+1+\epsilon}, 
\end{equation}
\item if $\mu_{i}>0$ and $\nu_{i}\in(0,1]$ are defined by
\[
\vert P_{i}(\zeta)\vert = H(P_{i})^{-\mu_{i}}, \qquad H(Q_{i})^{\nu_{i}}=H(P_{i})
\]
then for large $i\geq i_{1}(\epsilon)$ we have
\begin{equation}  \label{eq:iden}
\frac{n}{\nu_{i}(\mu_{i}-n+1)}<1+\epsilon.
\end{equation}
\end{itemize}
Then for large $i\geq i_{2}(\epsilon)$ and the parameter $X_{i}=H(Q_{i})$, the system
\begin{equation} \label{eq:nolongere}
H(P)\ll_{n} X_{i}, \qquad 0<\vert P(\zeta)\vert \leq X_{i}^{-2n+1-\epsilon}
\end{equation}
has no non-trivial solution $P\in\mathbb{Z}[T]$ of degree at most $2n-1$,
for a suitable implied constant. In particular
\begin{equation} \label{eq:nolonger}
\widehat{w}_{2n-1}(\zeta)=2n-1, \qquad \widehat{\lambda}_{2n-1}(\zeta)=\frac{1}{2n-1}.
\end{equation}
\end{theorem}

\begin{remark}
If we weaken the assumptions by only prescribing some upper bound less than $2n-1$ 
for $w_{n-1}(\zeta)$ and/or replacing the right hand side in \eqref{eq:iden} by a larger number,
we still obtain upper bounds for $\widehat{w}_{2n-1}(\zeta)$ and $\widehat{\lambda}_{2n-1}(\zeta)$.
Moreover, upon the assumptions of the theorem
the condition \eqref{eq:iden} in fact implies that the 
left hand side of \eqref{eq:iden} must
tend to $1$ as $i\to\infty$. We could equivalently
impose $\vert n\nu_{i}^{-1}(\mu_{i}-n+1)^{-1}-1\vert<\epsilon$
for $i\geq i_{0}(\epsilon)$ instead of \eqref{eq:iden}.
Furthermore the left condition in \eqref{eq:fra} is
in fact redundant as it can be derived from the right condition
in \eqref{eq:fra} combined with \eqref{eq:iden}. 
We do not carry these issues out. 
\end{remark}

We will now carry out how to apply Theorem~\ref{vorb} to 
Sturmian words $\zeta=\zeta_{\varphi}$ and $n=2$   
to deduce \eqref{eq:glm}. 
The method generalizes the proof of~\cite[(12) in Theorem~2.1]{ichlondon} and leads to
a deeper insight why it is true.
We have to check that the assumptions of Theorem~\ref{vorb} are satisfied in this context.
This follows from certain results due to Bugeaud and Laurent~\cite{buglau}.
First we introduce some additional notation.
Let $\zeta_{\varphi}$ be defined  as in \eqref{eq:s} via 
a sequence $(s_{k})_{k\geq 1}$
and let $(m_{k})_{k\geq 1}$ be the corresponding words as in \eqref{eq:m}. Then for $k\geq 1$ let
\begin{equation} \label{eq:eta}
\eta_{k}:=[s_{k+1};s_{k},\ldots,s_{1}], \qquad \alpha_{k}:=[0;m_{k},m_{k},\ldots].
\end{equation}
Every $\alpha_{k}$ is a quadratic irrational number.
Let $W_{k}$ denote its minimal polynomial over $\mathbb{Z}[T]$ with coprime coefficients
and let $H(\alpha_{k})=H(W_{k})$.
Concretely the following was shown in the proofs of~\cite[Corollary~6.1 and Corollary~6.2]{buglau}. 
As usual $a\asymp b$ means
both $a\ll b$ and $b\ll a$ are satisfied everywhere it occurs in the sequel.

\begin{lemma}[Bugeaud, Laurent]  \label{laubug}
Let $\zeta_{\varphi}$ be as above. 
Then
\begin{equation} \label{eq:abstand}
\vert \zeta_{\varphi}-\alpha_{k}\vert \asymp H(\alpha_{k})^{-2-2\eta_{k}},
\end{equation}
and
\begin{equation} \label{eq:jogenau}
\vert W_{k}(\zeta_{\varphi})\vert \asymp H(W_{k})^{-1-2\eta_{k}}.
\end{equation}
Moreover
\[
\vert \zeta_{\varphi}-\alpha_{k}\vert \asymp H(\alpha_{k})^{-1} H(\alpha_{k+1})^{-2-\frac{1}{\eta_{k}}},
\]
and 
\begin{equation} \label{eq:sieg}
\vert W_{k}(\zeta_{\varphi})\vert \asymp H(W_{k+1})^{-2-\frac{1}{\eta_{k}}}.
\end{equation}
\end{lemma} 

Since $\eta_{k}$ corresponds to $1/\sigma$ by \eqref{eq:sigmarder},
Lemma~\ref{laubug} readily yields the lower bounds for the exponents in Theorem~\ref{buglauthm}.
For the upper bounds Liouville inequality was applied to exclude better quadratic approximations, 
see the proofs of~\cite[Corollary~6.1 and Corollary~6.2]{buglau}. It was shown that
there are no algebraic $\alpha$ of degree at most two and height less than $H(\alpha_{k+1})$
closer to $\zeta$ than $\alpha_{k}$. Similarly $\vert P(\zeta_{\varphi})\vert$ is minimized
among all linear or quadratic $P\in\mathbb{Z}[T]$ of height less than $H(W_{k+1})$ for $P=W_{k}$.
We will apply this in the proof of Theorem~\ref{sternchen}. In fact 
the proof of Theorem~\ref{neu} will confirm it, though.
We show that the conditions of Theorem~\ref{vorb}
are satisfied for $\zeta_{\varphi}$ with the polynomials given by
\[
P_{i}= W_{i}, \qquad Q_{i}=W_{i+1}, 
\]
with $W_{i}$ as above. 
Observe that these specializations lead to
\begin{equation} \label{eq:obenauf}
\mu_{i}=1+2\eta_{i}+o(1), \qquad \nu_{i}=\frac{2+\frac{1}{\eta_{i}}}{1+2\eta_{i}}+o(1)=\frac{1}{\eta_{i}}+o(1) 
\end{equation}
by \eqref{eq:jogenau} and \eqref{eq:sieg}. 
It follows from $\eta_{k}>1$ that \eqref{eq:fra} is satisfied and
\eqref{eq:iden} can be readily checked for $n=2$ with \eqref{eq:obenauf} as well. 
The remaining conditions are obviously satisfied.
Thus Theorem~\ref{vorb} implies the following.

\begin{corollary} \label{folgt}
Let $\zeta=\zeta_{\varphi}$ be as above. Then for every $\epsilon>0$ 
and large $i\geq i_{0}(\epsilon)$ with $X_{i}=H(W_{i})$ the inequality 
\eqref{eq:nolongere} for $n=2$ has no solution $P\in\mathbb{Z}[T]$ of degree at most $3$. In particular we deduce \eqref{eq:glm} 
a consequence of Theorem~\ref{vorb}.
\end{corollary}

For the proof \eqref{eq:equ} in Section~\ref{proofs} we will
recall the general assertion of Corollary~\ref{folgt}, in
particular that $X_{i}$ for the conclusion \eqref{eq:glm}
can be chosen of the special form $H(W_{i})$.
In fact we will need the following fact evolving from 
the proof of Theorem~\ref{vorb} in combination with 
the above observations: for $\epsilon>0$ and 
large $i\geq i_{0}(\epsilon)$ there exist $R_{1}(T),R_{2}(T)$ linear integer polynomials such that 
the derived polynomials 
\[
\{G_{1,i},G_{2,i},G_{3,i},G_{4,i}\}=\{R_{1}W_{i},R_{2}W_{i},W_{i+1},TW_{i+1}\}
\]
span the four-dimensional space of polynomials of degree at most three,
and with $X_{i}:=H(W_{i+1})$ satisfy
\[
\max_{1\leq j\leq 4} H(G_{j,i})\ll X_{i}, \qquad \max_{1\leq j\leq 4} \vert G_{j,i}(\zeta)\vert \leq X_{i}^{-3+\epsilon}.
\]
The proof of \eqref{eq:equ} requires some more preparation from Section~\ref{para} below. 

\subsection{An identity for Sturmian continued fractions}

At this point we want to enclose some remarks on the proof in the last section.
A special case of~\cite[Theorem~2.2]{buschl} shows that for any real number $\zeta$
which satisfies the condition $w_{n}(\zeta)> w_{n-1}(\zeta)$, we have
\begin{equation} \label{eq:varphi}
\widehat{w}_{n}(\zeta)\leq \frac{nw_{n}(\zeta)}{w_{n}(\zeta)-n+1}. 
\end{equation}
We do not know whether the assumption $w_{n}(\zeta)> w_{n-1}(\zeta)$ is really necessary 
for the conclusion or not, the analogue of \eqref{eq:varphi} for the exponents 
$w_{n}^{\ast}, \widehat{w}_{n}^{\ast}$ holds unconditionally~\cite[Theorem~2.4]{buschl}.
See also~\cite{ichcommapp} for an unconditioned weaker estimate in the flavor of \eqref{eq:varphi}.
Inserting the values of Theorem~\ref{buglauthm}, we see that for $n=2$ the Sturmian continued fractions $\zeta_{\varphi}$
provide equality in \eqref{eq:varphi}. Our proof of~\eqref{eq:glm} relied on a parametric version of this identity,
concretely \eqref{eq:iden}, that arose from Lemma~\ref{laubug}.
Indeed, if we identify $\mu_{i}$ with $w_{2}(\zeta_{\varphi})$ and
$\nu_{i}$ with $\widehat{w}_{2}(\zeta_{\varphi})/w_{2}(\zeta_{\varphi})$,
then \eqref{eq:iden} becomes the reverse inequality
to \eqref{eq:varphi} as $\epsilon\to 0$. We see that in fact 
the quotients $n/(\nu_{i}(\mu_{i}-n+1))$ in \eqref{eq:iden} converge to $1$ as $i\to\infty$ for $\zeta_{\varphi}$. 

We briefly discuss equality in \eqref{eq:varphi} for $n>2$.
By \eqref{eq:wmono}, equality obviously occurs for any $U_{n}$-number in Mahler's classification of real numbers, 
that is a transcendental real number which
satisfies $w_{n}(\zeta)=\infty$ and $n$ is the smallest such index. 
The existence
of $U_{n}$-numbers of any prescribed degree $n\geq 1$ was constructively proved by LeVeque~\cite{leveque}.
Most likely equality will hold for all (or at least some) $U_{m}$-numbers with $m<n$ as well. 
For $m=1$ this is true as a consequence of~\cite[Theorem~1.12]{schlei}.
However, for $n\geq 3$ and $w_{n}(\zeta)<\infty$ (or $\widehat{w}_{n}(\zeta)>n$),
it is completely open whether equality in \eqref{eq:varphi} can occur.
In fact it is not even known if there exist transcendental real
$\zeta$ with $\widehat{w}_{n}(\zeta)>n$ for some $n\geq 3$.
A necessary condition for equality in \eqref{eq:varphi} is $w_{n}(\zeta)\geq h_{n}$, where 
$h_{n}>2n-1$ is given by
\[
h_{n}:= \frac{1}{2}\left(\frac{1+2n\sqrt{n^2-2n+\frac{5}{4}}}{n-1}+2n-1 \right),
\]
and even a slightly larger bound can be given when $n=3$. 
This follows from the proof of~\cite[Theorem~2.1]{buschl} based on the fact that otherwise smaller upper bounds
for $\widehat{w}_{n}(\zeta)$ are obtained by Schmidt and Summerer~\cite{ssch}.
Assuming a conjecture of Schmidt and Summerer~\cite{sums} the bound $h_{n}$ can be replaced 
by a larger one in case of $n\geq 4$ as well. 
See~\cite[Section~5.1]{ichglasgow}, in particular~\cite[(33)]{ichglasgow}, 
on how to derive concrete numerical bounds smaller than $h_{n}$. 

\section{Parametric geometry of numbers} \label{para}

We recall the concepts of the parametric geometry of numbers for the proof of \eqref{eq:equ}.
We follow Schmidt and Summerer~\cite{ss},~\cite{ssch}, very 
similarly as in~\cite{ichlondon}.
Let $\zeta\in{\mathbb{R}}$ and an integer $n\geq 1$ be given, and let
$Q>1$ be a parameter. For $1\leq j\leq n+1$,
define $\psi_{n,j}(Q)$ as the minimum of $\eta\in{\mathbb{R}}$ such that
\[
\vert x\vert \leq Q^{1+\eta}, \qquad \max_{1\leq j\leq n} \vert \zeta^{j}x-y_{j}\vert\leq Q^{-\frac{1}{n}+\eta} 
\]
has (at least) $j$ linearly independent solutions in integer 
vectors $(x,y_{1},\ldots,y_{n})$. It is easy to verify 
that they are non-decreasing in the second parameter and
the possible range of every $\psi_{n,j}$ for $Q\in(1,\infty)$ is given by
\[
-1\leq \psi_{n,1}(Q)\leq \psi_{n,2}(Q)\leq \cdots \leq
\psi_{n,n+1}(Q)\leq
\frac{1}{n}.
\]
Define the derived quantities  
\[
\underline{\psi}_{n,j}=\liminf_{Q\to\infty} \psi_{n,j}(Q),
\qquad \overline{\psi}_{n,j}=\limsup_{Q\to\infty} \psi_{n,j}(Q).
\]
These quantities obviously all lie in the interval $[-1,1/n]$. In this setting,
Dirichlet's Theorem becomes
$\psi_{n,1}(Q)<0$ for all $Q>1$, thus $\overline{\psi}_{n,1}\leq 0$.
Now we define the dual functions $\psi_{n,j}^{\ast}(Q)$
from~\cite{ss}. For $1\leq j\leq n+1$ and a parameter $Q>1$, 
let $\psi_{n,j}^{\ast}(Q)$ 
be the minimum of $\eta\in{\mathbb{R}}$ such that
\[
H(P) \leq Q^{\frac{1}{n}+\eta}, \qquad 
\vert P(\zeta)\vert\leq Q^{-1+\eta} 
\]
has (at least) $j$ linearly independent solutions in 
polynomials $P\in\mathbb{Z}[T]$ of degree at most $n$. 
The range for these functions is given by
\begin{equation} \label{eq:linkda}
-\frac{1}{n}\leq \psi_{n,1}^{\ast}(Q)\leq \psi_{n,2}^{\ast}(Q)
\leq \cdots \leq \psi_{n,n+1}^{\ast}(Q)\leq 1.
\end{equation}
Again we consider
\[
\underline{\psi}_{n,j}^{\ast}=\liminf_{Q\to\infty} \psi_{n,j}^{\ast}(Q),
\qquad \overline{\psi}_{n,j}^{\ast}=\limsup_{Q\to\infty} \psi_{n,j}^{\ast}(Q),
\]
which take values in $[-1/n,1]$. 
Moreover $\overline{\psi}_{n,1}^{\ast}\leq 0$ follows again
from Dirichlet's Theorem.
For transcendental real $\zeta$, Schmidt 
and Summerer~\cite[(1.11)]{ssch} established the inequalities 
\[
j\underline{\psi}_{n,j}+(n+1-j)\overline{\psi}_{n,n+1}\geq 0, 
\qquad j\underline{\psi}_{n,j}+(n+1-j)\overline{\psi}_{n,n+1}\geq 0,
\]
for $1\leq j\leq n+1$. The dual inequalities
\begin{equation} \label{eq:duales}
j\underline{\psi}_{n,j}^{\ast}+(n+1-j)\overline{\psi}_{n,n+1}^{\ast}\geq 0, 
\qquad j\underline{\psi}_{n,j}^{\ast}+(n+1-j)\overline{\psi}_{n,n+1}^{\ast}\geq 0,
\end{equation}
hold as well for the same reason, as pointed out in~\cite{ichlondon}.
Mahler's duality implies
\begin{equation} \label{eq:jaja}
\underline{\psi}_{n,j}=-\overline{\psi}_{n,n+2-j}^{\ast}, \qquad 
\overline{\psi}_{n,j}=-\underline{\psi}_{n,n+2-j}^{\ast},
\qquad 1\leq j\leq n+1.
\end{equation}
For $q=\log Q>0$ we further define the derived functions
\begin{equation} \label{eq:lpsi}
L_{n,j}(q)= q\psi_{n,j}(Q), \qquad L_{n,j}^{\ast}(q)= q\psi_{n,j}^{\ast}(Q).
\end{equation}
These functions are piecewise linear with slopes among $\{-1,1/n\}$ and $\{-1/n,1\}$ respectively.
Further for any $\underline{x}=(x,y_{1},\ldots,y_{n})\in\mathbb{Z}^{n+1}$ define the function
\begin{equation} \label{eq:netto}
L_{\underline{x}}(q)=\max \left\{ \log \vert x\vert-q, 
\max_{1\leq j\leq n} \log \vert \zeta^{j}x-y_{j}\vert+\frac{q}{n}\right\}.
\end{equation}
Then by construction $L_{n,j}(q)=L_{\underline{x}}(q)$ for every $q>0$ and some uniquely defined $\underline{x}=\underline{x}(j,q)$.
In case of $j=1$, the vector $\underline{x}$ is chosen such that the expression in \eqref{eq:netto} is minimized among all non-zero vectors.
Similarly, any $L_{n,j}^{\ast}$ coincides at any point $q$ with some
\begin{equation} \label{eq:incide}
L_{P}^{\ast}(q)=\max \left\{ \log H(P)-\frac{q}{n}, 
\log \vert P(\zeta)\vert+q \right\}  
\end{equation}
for some $P\in\mathbb{Z}[T]$ of degree at most $n$.
Minkowski's second lattice point Theorem translates into
\begin{equation} \label{eq:lsumme}
\left\vert \sum_{j=1}^{n+1} L_{n,j}(q)\right\vert \ll 1, \qquad 
\left\vert \sum_{j=1}^{n+1} L_{n,j}^{\ast}(q)\right\vert \ll 1.
\end{equation}
Hence in any interval $I=(q_{1},q_{2})$, the sums of the differences
$L_{n,j}(q_{2})-L_{n,j}(q_{1})$ and $L_{n,j}^{\ast}(q_{2})-L_{n,j}^{\ast}(q_{1})$ over $1\leq j\leq n+1$,
are bounded in absolute value.

We relate $\psi_{n,j}(Q), \psi_{n,j}^{\ast}(Q)$ and the quantities derived from them
to another kind of successive minima approximation exponents,
which generalize the exponents $w_{n}, \widehat{w}_{n}, \lambda_{n}, \widehat{\lambda}_{n}$ from Section~\ref{sek1}.
Fix $n\geq 1$ and let $1\leq j\leq n+1$. Let $w_{n,j}(\zeta)$ and $\widehat{w}_{n,j}(\zeta)$ be the supremum of  
$w\in{\mathbb{R}}$ such that \eqref{eq:w}
has (at least) $j$ linearly independent polynomial solutions 
$P(T)\in\mathbb{Z}[T]$ of degree at most $n$
for arbitrarily large $X$, and all large $X$, respectively.
Similarly, define $\lambda_{n,j}(\zeta)$ and $\widehat{\lambda}_{n,j}(\zeta)$ as
the supremum of $\lambda\in{\mathbb{R}}$ such that \eqref{eq:lambda}
has (at least) $j$ linearly independent solutions $(x,y_{1},y_{2},\ldots, y_{n})\in{\mathbb{Z}^{n+1}}$ 
for arbitrarily large $X$, and all large $X$, respectively. For $j=1$
just the classic exponents are obtained.
The indicated relation between 
$\psi_{n,j}(Q), \psi_{n,j}^{\ast}(Q)$ 
and the above successive minima exponents essentially
established in~\cite{ss}
is given by
\begin{equation} \label{eq:umrechnen}
(1+\lambda_{n,j}(\zeta))(1+\underline{\psi}_{n,j})=
(1+\widehat{\lambda}_{n,j}(\zeta))(1+\overline{\psi}_{n,j})=\frac{n+1}{n}, \qquad 1\leq j\leq n+1,
\end{equation}
and 
\begin{equation} \label{eq:umrechnen2}
(1+w_{n,j}(\zeta))\Big(\frac{1}{n}+\underline{\psi}_{n,j}^{\ast}\Big)=
(1+\widehat{w}_{n,j}(\zeta))\Big(\frac{1}{n}+\overline{\psi}_{n,j}^{\ast}\Big)=\frac{n+1}{n}, \qquad 1\leq j\leq n+1.
\end{equation}
As already quoted in~\cite{ichlondon} from \eqref{eq:jaja}, \eqref{eq:umrechnen} and \eqref{eq:umrechnen2} 
one may deduce
\begin{equation} \label{eq:debre}
\lambda_{n,j}(\zeta)= \frac{1}{\widehat{w}_{n,n+2-j}(\zeta)}, \qquad
\widehat{\lambda}_{n,j}(\zeta)= \frac{1}{w_{n,n+2-j}(\zeta)}, \qquad\quad  1\leq j\leq n+1.
\end{equation}
We recall~\cite[Lemma~3.3]{ichlondon} in a slightly modified form.
\begin{lemma} \label{technisch}
Let $\zeta$ be a real transcendental number.
Let $P,Q\in\mathbb{Z}[T]$ be of large heights, $R=PQ$ and suppose $R$ has degree at most three. 
Define the functions $L_{P}^{\ast}, L_{R}^{\ast}$
as in \eqref{eq:incide} with respect to $n=3$.
Assume that $(q_{1},L_{P}^{\ast}(q_{1}))$ is the local minimum of 
$L_{P}^{\ast}$ and $(q_{2},L_{R}^{\ast}(q_{2}))$ the local minimum 
of $L_{R}^{\ast}$.
Further assume 
\begin{equation} \label{eq:nicht}
\vert Q(\zeta)\vert\asymp H(Q)^{-1}.
\end{equation}
Then, we have
\begin{equation} \label{eq:convers}
\frac{L_{R}^{\ast}(q_{2})-L_{P}^{\ast}(q_{1})}{q_{2}-q_{1}}=\frac{1}{3}+o(1),  \qquad\qquad H(Q)\to\infty.
\end{equation}
\end{lemma}
We omit the proof since it is very similar to the one 
of~\cite[Lemma~3.3]{ichlondon}. Our error term is even smaller 
as a consequence
of the stronger condition \eqref{eq:nicht} instead
of $\vert Q(\zeta)\vert\leq  H(Q)^{-1+\delta}$ for small $\delta$. 
The stronger condition will be satisfied in our applications
to the numbers $\zeta_{\varphi}$ for the linear best approximation
polynomials $Q$, since $\zeta_{\varphi}$ are
badly approximable with respect to one-dimensional approximation. 
This property is unknown for extremal numbers,
only the weaker claim
that $\vert\zeta-y/x\vert\gg x^{-2} \max\{2,\log x\}^{-c}$ for some $c=c(\zeta)\geq 0$ 
and all rational numbers $y/x$ was shown by Roy~\cite[Theorem~1.3]{royyy}.

\section{Proofs of Theorems~\ref{vorb},~\ref{neu} and \ref{sternchen}} \label{proofs}

For convenience we first provide an easy proposition, which 
in fact is an immediate consequence 
of Minkowski's second lattice point theorem.

\begin{proposition} \label{leicht}
Let $k\geq 1$ be an integer and $\zeta$ be a real number. Then
the following assertions are equivalent.
\begin{itemize}
\item $w_{k,k+1}(\zeta) \geq k$ 
\item $w_{k,k+1}(\zeta)=k$
\item $\widehat{w}_{k}(\zeta)\leq k$
\item $\widehat{w}_{k}(\zeta)= k$
\item $\widehat{\lambda}_{k}(\zeta)= \frac{1}{k}$
\end{itemize}

\end{proposition}

\begin{proof}
It is obvious by \eqref{eq:wmono} that the last two claims are equivalent. From the right relation in
\eqref{eq:debre} we see that $w_{k,k+1}(\zeta)=\widehat{\lambda}_{k}(\zeta)^{-1}$, and together with \eqref{eq:ldiri}
the first two and the last claims are all
equivalent. To finish the proof, it suffices to show that the first
and third relation are equivalent. This is essentially
an application of
Minkowski's second Theorem. If $w_{k,k+1}(\zeta)\geq k$ holds, then 
\eqref{eq:umrechnen2} implies $\underline{\psi}_{k,k+1}^{\ast}\leq 0$.
Hence for certain arbitrarily large $Q$ we have 
$\psi_{k,k+1}^{\ast}(Q)\leq \epsilon$.
It follows from \eqref{eq:linkda}, \eqref{eq:lpsi} and \eqref{eq:lsumme}
that $\psi_{k,1}^{\ast}(Q)\geq -k\epsilon+o(1)$
as $Q\to\infty$. Hence $\overline{\psi}_{k,1}^{\ast}\geq 0$
as $\epsilon$ can be chosen arbitrarily small
(and thus actually $\overline{\psi}_{k,1}^{\ast}= 0$). From
\eqref{eq:umrechnen2} with $j=1$ we further  
deduce $\widehat{w}_{k}(\zeta)\leq k$
as claimed.
The reverse implication is performed by reversing the argument, 
we leave the details to the reader.
%
\end{proof}

We further
recall an estimate often referred to as Gelfond's lemma,
see also~\cite[Hilfssatz~3]{wirsing}. It asserts that
for polynomials $Q_{1},Q_{2}$ with integral coefficients of degree at most $n$, we have
\begin{equation} \label{eq:multipli}
H(Q_{1}Q_{2})\asymp_{n} H(Q_{1})H(Q_{2}).
\end{equation}
Now we can prove Theorem~\ref{vorb}. 

\begin{proof}[Proof of Theorem~\ref{vorb}]
In view of Proposition~\ref{leicht} with $k=2n-1$,
it suffices to show \eqref{eq:nolongere}. Further by Proposition~\ref{leicht}
it suffices to find for certain arbitrarily large real $X_{i}$ a set of
$2n$ linearly independent polynomials $G_{1,i},\ldots,G_{2n,i}\in\mathbb{Z}[T]$ that satisfy
\begin{equation} \label{eq:sfy}
H(G_{j,i})\leq X_{i}, \qquad \vert G_{j,i}(\zeta)\vert \leq X_{i}^{-2n+1+\epsilon},
\qquad 1\leq j\leq 2n, \; i\geq 1.
\end{equation}

Since by assumption $w_{n-1}(\zeta)=n-1<n\leq w_{n}(\zeta)$, for large $i$
the polynomials $P_{i}$ and $Q_{i}$ must have degree 
precisely $n$ and be irreducible. This follows from \eqref{eq:multipli} and the definition
of $w_{n}$ and was essentially carried out in~\cite{wirsing} and already used frequently in~\cite{buschl}.
Let $X_{i}:=H(Q_{i})$. By assumption
\[
H(P_{i})=X_{i}^{\nu_{i}}, \qquad \vert P_{i}(\zeta)\vert =X_{i}^{-\nu_{i}\mu_{i}}.
\]
%
By Proposition~\ref{leicht} with $k=n-1$, our assumption $w_{n-1}(\zeta)=n-1$
yields $\widehat{w}_{n-1,n}(\zeta)=n-1$. In other words for any large parameter $Y\geq Y_{0}(\epsilon)$
there exist linearly independent polynomials $R_{1},\ldots,R_{n}\in\mathbb{Z}[T]$ 
of degree at most $n-1$ that satisfy
\[
H(R_{j})\leq Y, \qquad \vert R_{j}(\zeta)\vert\leq Y^{-n+1+\epsilon}, \qquad 1\leq j\leq n.
\]
Choose $Y=Y_{i}$ with the sequence of parameters
\[
Y_{i}=\frac{X_{i}}{H(P_{i})}= X_{i}^{1-\nu_{i}}.
\]
For $1\leq j\leq n$ and $i\geq 1$ denote by $R_{j,i}$ the polynomials $R_{j}$ above for the parameter $Y_{i}$.
We show that the polynomials $P_{j,i}:= P_{i}R_{j,i}$ 
provide $n$ of the $2n$ polynomials in \eqref{eq:sfy}
with respect to the parameter $X_{i}$. Clearly each $P_{j,i}$ has degree at most $n+(n-1)=2n-1$.
Moreover by construction and \eqref{eq:multipli} we have
\[
H(P_{j,i})\ll_{n} H(P_{i})\cdot H(R_{j,i})\leq X_{i}, \qquad\qquad  1\leq j\leq n,\; i\geq 1.
\]
The evaluations of the $P_{j,i}$ at $\zeta$ can be estimated via
\begin{equation}  \label{eq:fuchs}
\vert P_{j,i}(\zeta)\vert= \vert P_{i}(\zeta)\vert\cdot \vert R_{j,i}(\zeta)\vert
\leq X_{i}^{-\nu_{i}\mu_{i}}\cdot X_{i}^{(1-\nu_{i})(-n+1+\epsilon)}, \qquad 1\leq j\leq n,\; i\geq i_{1}.
\end{equation}
By assumption \eqref{eq:iden} of the theorem we infer
\[
\vert P_{j,i}(\zeta)\vert \leq X_{i}^{-2n+1+\epsilon_{0}}, \qquad\qquad 1\leq j\leq n,\; i\geq i_{1},
\]
for
$\epsilon_{0}=n-\frac{n}{1+\epsilon}+\epsilon(1-\nu_{i})$,
which tends to $0$ as $\epsilon$ does.
For the remaining $n$ polynomials we 
take $Q_{j,i}(T)=T^{j-1}Q_{i}(T)$ for $1\leq j\leq n$ and $i\geq 1$.  
Obviously $Q_{j,i}$ have degree at most $2n-1$, height $H(Q_{j,i})=H(Q_{i})=X_{i}$ and 
satisfy $\vert Q_{j,i}(\zeta)\vert \ll_{\zeta,n} \vert Q_{i}(\zeta)\vert$ for $1\leq j\leq n$. Hence 
\[
\vert Q_{j,i}(\zeta)\vert \ll_{n,\zeta} X_{i}^{-2n+1+\epsilon}, \qquad 1\leq j\leq n,\; i\geq i_{1}.
\]
Let $\epsilon_{1}=\max\{ \epsilon,\epsilon_{0}\}$. Summing up, 
for certain arbitrarily large $X=X_{i}$ 
we have found $2n$ polynomials 
$\mathcal{G}:=\{G_{1,i},\ldots,G_{2n,i}\}:=\{P_{1,i},\ldots,P_{n,i},Q_{1,i},\ldots,Q_{n,i}\}$
that satisfy 
\[
H(G_{j,i})\ll_{n} X_{i}, \qquad \vert G_{j,i}(\zeta)\vert \leq X_{i}^{-2n+1+\epsilon_{1}}
\qquad\qquad 1\leq j\leq 2n, \; i\geq i_{2}.
\]

As $\epsilon_{1}$ clearly can be chosen arbitrarily small,
for the proof of \eqref{eq:sfy} it remains to be checked that 
$\mathcal{G}$ is a linearly independent set. The linear independence of 
$\mathcal{P}:=\{P_{1,i},\ldots,P_{n,i}\}$ follows obviously from the linear 
independence of $\{R_{1,i},\ldots,R_{n,i}\}$.
The linear independence of $\mathcal{Q}:=\{Q_{1,i},\ldots,Q_{n,i}\}$ is obvious.
Any linear combination of the $P_{j,i}$ is of the form 
$P_{i}U$ for some polynomial $U\in\mathbb{Z}[T]$ and any linear combination of the $Q_{j,i}$ is of the form $Q_{i}V$ for 
some $V\in\mathbb{Z}[T]$ where $U,V$ have
degree at most $n-1$. Since on the other hand $P_{i},Q_{i}$ are 
linearly independent, irreducible and of degree $n$,
by the unique factorization in $\mathbb{Z}[T]$ we infer that the
equation $P_{i}U+Q_{i}V=0$ can only be satisfied if both $U$ and $V$ vanish identically.
By the linear independence of $\mathcal{P}$ and $\mathcal{Q}$
this results in the linear independence of $\mathcal{G}$. The proof is finished.
\end{proof}

Now we turn to the proof of the following next partial assertion of Theorem~\ref{neu}.
\begin{theorem} \label{rest}
Let $\zeta=\zeta_{\varphi}$ with corresponding $\sigma=\sigma_{\varphi}$ be as above. Then 
\[
w_{3}(\zeta)=1+\frac{2}{\sigma}, \qquad \lambda_{3}(\zeta)= \frac{1}{1+2\sigma}.
\]
\end{theorem}

We provide a brief outline and preliminaries
of the proof of Theorem~\ref{rest}.
We will establish a rather precise description of the functions
$L_{3,1}^{\ast}(q),\ldots,L_{3,4}^{\ast}(q)$ on $q\in(0,\infty)$ 
induced by $(\zeta_{\varphi},\zeta_{\varphi}^{2},\zeta_{\varphi}^{3})$.
Denote by $\vert I\vert$ the length of an interval $I$. We will show
that any $\zeta_{\varphi}$ induces a partition 
of the positive real numbers in half-open successive intervals $I_{1},J_{1},I_{2},J_{2},\ldots$ 
with the following properties.
  
\begin{itemize}
\item $\lim_{k\to\infty} \vert I_{k}\vert/\vert J_{k}\vert=1$.
\item At the left interval end of every $I_{k}$, all $L_{3,i}^{\ast}(q)$ are small (more precisely
$o(q)$ as $q\to\infty$) by absolute value. In $I_{k}$ the
functions $L_{3,1}^{\ast}(q), L_{3,2}^{\ast}(q)$ essentially decay with slope $-1/3$,
whereas $L_{3,3}^{\ast}(q),L_{3,4}^{\ast}(q)$ essentially rise with slope $1/3$.
\item At the right interval end of $I_{k}$, which equals the left interval end
of $J_{k}$, the opposite behavior appears. The functions $L_{3,1}^{\ast}(q), L_{3,2}^{\ast}(q)$ essentially
rise with slope $1/3$ in $J_{k}$,
whereas $L_{3,3}^{\ast}(q), L_{3,4}^{\ast}(q)$ essentially decay with slope $-1/3$. Ultimately
the functions $L_{3,1}^{\ast},\ldots,L_{3,4}^{\ast}$ asymptotically meet at the right end of $J_{k}$, 
which equals the left interval end of the successive $I_{k+1}$. 
\item The functions $\vert L_{3,1}^{\ast}(q)-L_{3,2}^{\ast}(q)\vert$, 
$\vert L_{3,3}^{\ast}(q)-L_{3,4}^{\ast}(q)\vert$ are of order $o(q)$
as $q\to\infty$
\end{itemize}

The word ''essentially'' in the above description means that
the stated behavior might be violated in short intervals only.
As for the special case of extremal numbers $\zeta_{\gamma-1}$ in~\cite{ichlondon}, 
the description applies to the simultaneous approximation functions $L_{3,j}(q)$ as well 
by \eqref{eq:jaja}, and as for the special case of extremal numbers 
in~\cite{ichlondon} we have
\begin{align}
w_{3,1}(\zeta)&=w_{3,2}(\zeta), \quad w_{3,3}(\zeta)= w_{3,4}(\zeta), \quad
\widehat{w}_{3,1}(\zeta)=\widehat{w}_{3,2}(\zeta), 
\quad \widehat{w}_{3,3}(\zeta)=\widehat{w}_{3,4}(\zeta),     \label{eq:rumpel} \\
\lambda_{3,1}(\zeta)&=\lambda_{3,2}(\zeta), \quad \lambda_{3,3}(\zeta)= \lambda_{3,4}(\zeta), \quad
\widehat{\lambda}_{3,1}(\zeta)=\widehat{\lambda}_{3,2}(\zeta), 
\quad \widehat{\lambda}_{3,3}(\zeta)=\widehat{\lambda}_{3,4}(\zeta), \label{eq:stielzchen}
\end{align}
which refines the claim of Theorem~\ref{neu}. Similarly as in the special case,
the decay phases of $L_{3,i}^{\ast}$ will turn out to be induced by the polynomials
$W_{k}$ from Lemma~\ref{laubug}. The rising phases are induced by products $W_{k}E_{l}$
for fixed $W_{k}$ and suitable successive best approximation polynomials $E_{l}$ 
in dimension $1$, defined by the property
\begin{equation} \label{eq:ellute}
E_{l}(\zeta_{\varphi})= \min \{ \vert Q(\zeta_{\varphi})\vert: Q\in\mathbb{Z}[T], \; 
\deg(Q)=1,\; 1\leq H(Q)\leq H(E_{l})\}.
\end{equation}
By Dirichlet's Theorem and since on the other hand 
any $\zeta_{\varphi}$ is badly approximable,
for any parameter $X>1$ there exist an index $l$ such that
$H(E_{l})\leq H(E_{l-1})\leq X$ and 
\begin{equation} \label{eq:luzia}
\vert E_{l}(\zeta_{\varphi})\vert \asymp_{\zeta_{\varphi}} 
\vert E_{l-1}(\zeta_{\varphi})\vert  \asymp_{\zeta_{\varphi}} 
H(E_{l})^{-1} \asymp_{\zeta_{\varphi}} 
H(E_{l-1})^{-1}
\asymp_{\zeta_{\varphi}} X^{-1}.
\end{equation}
In fact one might take arbitrary many successive $l,l+1,\ldots,l+m$
with this property, but for us considering two
suffices.
In view of Lemma~\ref{technisch}, the relation
\eqref{eq:luzia} will lead to an
asymptotic increase by $1/3$ for graphs of $L_{3,j}^{\ast}$
induced by the succession of certain products 
$W_{k}E_{l}, W_{k}E_{l+1},\ldots,W_{k}E_{l+h}$.
In contrast to the special case of extremal 
numbers $\zeta_{\gamma-1}$, 
where the quotients of $\vert I_{k+1}\vert/\vert I_{k}\vert$
and $\vert J_{k+1}\vert/\vert J_{k}\vert$ tend to the 
golden ratio $\gamma$, in general they
depend on the sequence $(s_{k})_{k\geq 1}$ connected 
to $\zeta_{\varphi}$.
Our proof of Theorem~\ref{rest} below essentially employs the same method as for the determination 
of $\lambda_{3}(\zeta), w_{3}(\zeta)$ for extremal numbers $\zeta$ in the proof of~\cite[(11) in Theorem~2.1]{ichlondon}. However,
we attempt to provide clearer and better readable arguments
for some rather sketchy parts of the proof 
of~\cite[(11), Theorem~2.1]{ichlondon} below.
Recall the definitions of $\eta_{k}$ and $W_{k}$ from Section~\ref{uni} associated to $\zeta_{\varphi}$, 
the latter essentially replace the polynomials $P_{k}$ 
from~\cite{ichlondon}.
We will implicitly use the consequence of \eqref{eq:multipli}
that if a polynomial $Q$ of degree not exceeding $n$ factors 
as $Q=Q_{1}Q_{2}$, 
then $\vert Q(\zeta)\vert \leq H(Q)^{-w}$ 
implies that either $\vert Q_{1}(\zeta)\vert \ll_{n} H(Q_{1})^{-w}$
or $\vert Q_{2}(\zeta)\vert \ll_{n} H(Q_{2})^{-w}$. This type of argument
was already used by Wirsing~\cite{wirsing}. Since we deal with polynomials
of degree at most three, the implied constant is in fact absolute.

\begin{proof}[Proof of Theorem~\ref{rest}]
First we recall and justify the explanations at the end of
Section~\ref{glmres}. Here we denote the index by $k$
instead of $i$. In the proof of Theorem~\ref{vorb} 
with the auxiliary Proposition~\ref{leicht}, we saw
that for $P_{k},Q_{k}$ as given there, for the parameter
$X=H(Q_{k})$ we have four polynomials $G_{1,k},\ldots,G_{4,k}$
that satisfy $H(G_{i,k})\ll X, \vert G_{i,k}(\zeta)\vert\leq X^{-3+\epsilon}$. Furthermore the proof shows 
the $G_{i,k}$ can be chosen of the form 
\[
\{ G_{1,k},\ldots,G_{4,k}\}=\{ R_{1,k}\cdot P_{k},\; R_{2,k}\cdot P_{k},\; Q_{k},\; T\cdot Q_{k}\}
\]
for linear integer polynomials $R_{1,k}(T), R_{2,k}(T)$. 
In the present situation
$P_{k}$ corresponds to $W_{k}$ and $Q_{k}$ to $W_{k+1}$, as carried
out in Section~\ref{glmres}. Hence, for any large $k$, 
with $X_{k}=H(W_{k+1})$ we have 
four linearly independent polynomials 
\begin{equation} \label{eq:infact}
\mathscr{G}_{k}=\{ G_{1,k},\ldots,G_{4,k}\}=
\{ W_{k}\cdot R_{1,k}, W_{k}\cdot R_{2,k},W_{k+1},T\cdot W_{k+1}\},
\end{equation}
with $R_{1,k}(T),R_{2,k}(T)$ linear integer polynomials,
so that
\[
\max_{1\leq i\leq 4} H(G_{i,k})\ll X_{k}, 
\qquad  \max_{1\leq i\leq 4} \vert G_{i,k}(\zeta)\vert 
\leq X_{k}^{-3+\epsilon}.
\]
More precisely,
essentially the same argument as in the proof of 
Proposition~\ref{leicht} shows the following
parametric claim. For arbitrarily 
small $\varepsilon>0$ and sufficiently
large $k$ if we take
$q=q_{k}$ as the solution to
\[
L_{W_{k+1}}^{\ast}(q)= \log H(W_{k+1})-\frac{q}{3}=0,
\]
which leads to
\begin{equation}  \label{eq:qk}
q_{k}= 3 \log H(W_{k+1}),
\end{equation}
then 
\begin{equation} \label{eq:virtus}
\vert L_{3,i}^{\ast}(q_{k})\vert\leq \varepsilon q_{k}, \qquad 1\leq i\leq 4, 
\end{equation}
will hold for $\varepsilon$ a suitable
variation of $\epsilon$.
Indeed since $W_{k+1}$ is one of the involved polynomials
whose related function $L_{W_{k}}^{\ast}(q)$ decays at $q=q_{k}$,
in view of \eqref{eq:incide} the claim can be readily verified.

By Lemma~\ref{laubug}, with the constant
$\eta_{k}$ defined there, any polynomial $W_{k}$ satisfies
\begin{equation} \label{eq:brauchdoch}
-\frac{\log \vert W_{k}(\zeta)\vert}{\log H(W_{k})}= 1+2\eta_{k}+o(1)> 3,\qquad k\to\infty.
\end{equation}
Clearly the same argument applies to $TW_{k}$ as 
$H(TW_{k})=H(W_{k})$ and $\vert \zeta W_{k}(\zeta)\vert$
differs from $\vert W_{k}(\zeta)\vert$ only by the constant
factor $\zeta\neq 0$. 
Hence
\begin{equation} \label{eq:brauchtum}
-\frac{\log \vert G_{i,k}(\zeta)\vert}{\log H(G_{i,k})}
=1+2\eta_{k+1}+o(1)>3,\qquad 
i\in\{3,4\}, \; k\to\infty.
\end{equation}
Let $b_{k}$ and $c_{k}$ respectively
be the local minima of the functions $L_{G_{3,k}}^{\ast}(q),
L_{G_{4,k}}^{\ast}(q)$. In view of \eqref{eq:incide} 
and since $H(G_{3,k})=H(TG_{3,k})=H(G_{4,k})$ these
are given as the solutions to 
\[
\log H(W_{k+1})-\frac{b_{k}}{3}= \log \vert W_{k+1}(\zeta)\vert+b_{k},
\quad
\log H(W_{k+1})-\frac{c_{k}}{3}= \log \vert \zeta W_{k+1}(\zeta)\vert+c_{k},
\]
which yields
\begin{equation} \label{eq:vorhang}
b_{k}=\frac{3}{4}( \log H(W_{k+1})-\log\vert W_{k+1}(\zeta)\vert),\quad
c_{k}=\frac{3}{4}( \log H(W_{k+1})-\log\vert \zeta W_{k+1}(\zeta)\vert).
\end{equation}
With \eqref{eq:brauchtum} it is not hard to check that
\begin{equation} \label{eq:beschr}
q_{k}<\min\{ b_{k},c_{k}\}, \qquad \vert b_{k}-c_{k}\vert= O(1), \qquad 
\vert L_{G_{3,k}}^{\ast}(b_{k})-L_{G_{4,k}}^{\ast}(b_{k})\vert=O(1).
\end{equation}
Combining \eqref{eq:brauchtum}, \eqref{eq:vorhang} and
\eqref{eq:beschr},
a short calculation shows
\[
L_{G_{i,k}}^{\ast}(b_{k})= \log H(W_{k+1})-\frac{b_{k}}{3}+O(1)=
\left(\frac{1-\eta_{k+1}}{3(1+\eta_{k+1})}+o(1)\right)b_{k},\qquad 
i\in\{3,4\}, \; k\to\infty.
\]
Since
$G_{3,k}^{\ast}$ and $G_{4,k}^{\ast}$ are
linearly independent, 
in particular
\[
L_{3,1}^{\ast}(b_{k})\leq 
L_{3,2}^{\ast}(b_{k})\leq \left(\frac{1-\eta_{k+1}}{3(1+\eta_{k})}+o(1)\right)b_{k}.
\]
On the other hand, 
by \eqref{eq:virtus} we have 
$L_{3,2}^{\ast}(q_{k})\geq L_{3,1}^{\ast}(q_{k})\geq -\varepsilon q_{k}$.
Since we noticed $q_{k}<b_{k}$ and all functions
$L_{3,j}^{\ast}, 1\leq j\leq 4$, have slope at least $-1/3$, we have
$L_{3,1}^{\ast}(b_{k})\geq 
-\varepsilon q_{k}-\frac{1}{3}(b_{k}-q_{k})$. After a short calculation 
using \eqref{eq:qk}, \eqref{eq:brauchtum}, \eqref{eq:vorhang} and
\eqref{eq:beschr} again, we end up with the reverse 
asymptotic inequality
\[
L_{3,1}^{\ast}(b_{k})\geq \left(\frac{1-\eta_{k+1}}{3(1+\eta_{k+1})}
-\varepsilon\right)b_{k}.
\]
Thus, as we may let $\varepsilon\to 0$, we may write
\[
L_{3,1}^{\ast}(b_{k})=L_{3,2}^{\ast}(b_{k})
=\left(\frac{1-\eta_{k+1}}{3(1+\eta_{k+1})}+o(1)\right)b_{k},
\qquad\qquad k\to\infty.
\]
Equivalently, in terms of the functions $\psi^{\ast}_{3,.}$, 
upon defining
\[
B_{k}:=e^{b_{k}},  \qquad k\geq 1,
\]
we have the asymptotic formula
\begin{equation} \label{eq:brauche}
\psi_{3,j}^{\ast}(B_{k})
=\frac{1-\eta_{k+1}}{3(1+\eta_{k+1})}+o(1), \qquad j\in\{1,2\}, \; k\to\infty.  
\end{equation}
The above argument in fact shows that $L_{3,1}^{\ast}(q)$ and 
$L_{3,2}^{\ast}(q)$ both
decay with asymptotic slope $-1/3$ in the interval 
$I_{k}:=[q_{k},b_{k})$, that is for any $q_{k}\leq a<b\leq b_{k}$ 
we have
\begin{equation} \label{eq:jori}
L_{3,j}^{\ast}(b)-L_{3,j}^{\ast}(a)=(b-a)
\left(-\frac{1}{3}+o(1)\right)+O(1),\qquad
j\in\{1,2\}, \quad k\to\infty.
\end{equation}
Here and in similar succeeding estimates,
the additional $O(1)$ term is only needed
when $a$ an $b$ are roughly of the same size.
It follows from \eqref{eq:jori} and \eqref{eq:lsumme} that 
the sum $L_{3,3}^{\ast}+L_{3,4}^{\ast}$ asymptotically increases with constant
slope $2/3$ in $I_{k}$, that is for any $q_{k}\leq a<b\leq b_{k}$ we have
\begin{equation} \label{eq:slope}
L_{3,3}^{\ast}(b)+L_{3,4}^{\ast}(b)-
L_{3,3}^{\ast}(a)-L_{3,4}^{\ast}(a)= 
(b-a)\left(\frac{2}{3}+o(1)\right)+O(1), \qquad k\to\infty. 
\end{equation}
Assume we have already shown that both $L_{3,3}^{\ast}$ 
and $L_{3,4}^{\ast}$
increase at most by $1/3$ in the interval 
$(q_{k},\infty)$, that is for $q_{k}\leq b$ we have
\begin{equation} \label{eq:exakter}
L_{3,3}^{\ast}(b)-L_{3,3}^{\ast}(q_{k})\leq 
(b-q_{k})\left(\frac{1}{3}+\varepsilon\right), \quad 
L_{3,4}^{\ast}(b)-L_{3,4}^{\ast}(q_{k})\leq 
(b-q_{k})\left(\frac{1}{3}+\varepsilon\right).
\end{equation}
Then by \eqref{eq:slope} both must have asymptotically
slope $1/3+o(1)$ as $k\to\infty$ in the entire interval $I_{k}$, 
that is for $q_{k}\leq a<b\leq b_{k}$ we have
\begin{equation} \label{eq:balzzeit}
L_{3,j}^{\ast}(b)-L_{3,j}^{\ast}(a)=
(b-a)\left(\frac{1}{3}+o(1)\right)+O(1),\qquad
j\in\{3,4\}, \quad k\to\infty.
\end{equation}
To show \eqref{eq:exakter}, first
note that by \eqref{eq:virtus} and 
since $L_{3,4}^{\ast}\geq L_{3,3}^{\ast}$
it suffices to show
\begin{equation} \label{eq:showmal}
L_{3,4}^{\ast}(b)\leq 
(b-q_{k})\left(\frac{1}{3}+o(1)\right)+O(1),\qquad
\quad k\to\infty.
\end{equation}
For any parameter $Y_{k}>X_{k}=H(W_{k+1})$,
consider the two linear
polynomials $E_{t}, E_{t-1}$ defined by \eqref{eq:ellute},
with $t=t(Y_{k})$ chosen largest possible such that 
\[
\max \{ H(W_{k}E_{t}),H(W_{k}E_{t-1})\}\leq Y_{k}.
\]
By Gelfond's Lemma \eqref{eq:multipli} we 
have $H(E_{t})\asymp H(E_{t-1})\asymp Y_{k}/X_{k-1}$.
Thus and since $\zeta$ is badly approximable with respect to
one-dimensional approximation, as previously noticed in \eqref{eq:luzia}, we have
\begin{equation} \label{eq:et}
\vert E_{t}(\zeta)\vert \asymp_{\zeta} 
\vert E_{t-1}(\zeta)\vert \asymp_{\zeta} 
H(E_{t})^{-1} \asymp_{\zeta} 
H(E_{t-1})^{-1} \asymp_{\zeta}  
\left(\frac{Y_{k}}{X_{k-1}}\right)^{-1}=\frac{X_{k-1}}{Y_{k}}.
\end{equation}
Define $U_{k,1}:=W_{k}E_{t}$ and $U_{k,2}:=W_{k}E_{t-1}$ 
for $t=t(Y_{k})$,
depending on $Y_{k}$ which is treated as a variable.
As $Y_{k}$ increases in $(X_{k},\infty)$,
define the two functions $L_{f_{k,i}}^{\ast}$ for $i=1,2$
as the succession of the functions $L_{U_{k,i}}^{\ast}$ 
in $(q_{k},\infty)$ with the
respective polynomials $U_{k,i}=U_{k,i}(t)=U_{k,i}(Y_{k})$.
That is by \eqref{eq:incide} formally
\begin{equation}
L_{f_{k,1}}^{\ast}(q)= 
\min_{t\geq 1}\max\{ -\frac{q}{3}+\log H(W_{k}E_{t}),
q+\log \vert W_{k}(\zeta)E_{t}(\zeta)\vert\}  \label{eq:k+1}
\end{equation}
and if $t_{0}=t_{0}(q)$ denotes the minimum index
in \eqref{eq:k+1}, then
\begin{equation}
L_{f_{k,2}}^{\ast}(q)= \min_{ t\geq 1, t\neq t_{0}} 
\max\{ -\frac{q}{3}+\log H(W_{k}E_{t}),
q+\log \vert W_{k}(\zeta)E_{t}(\zeta)\vert\}.  \label{eq:k+2}
\end{equation}
In fact the latter minimum index $t_{1}$ is just $t_{1}=t_{0}-1$.
Since $H(E_{t})\ll_{\zeta} H(E_{t-1})$ by \eqref{eq:et},
application of Lemma~\ref{technisch} with 
the polynomials $P(T)=W_{k}(T)$, and $Q(T)=E_{t}(T)$
and $Q(T)=E_{t-1}(T)$ respectively, yields that
both induced
functions $L_{f_{k,1}}^{\ast},L_{f_{k,2}}^{\ast}$ rise 
with slope $1/3+o(1)$ 
in $(q_{k},\infty)$, as $H(E_{t})\to\infty$. Since
$H(E_{t})\asymp Y_{k}/X_{k-1}>X_{k}/X_{k-1}\to\infty$
as $k\to\infty$, it suffices to assume $k\to\infty$.
Hence for any $q_{k}\leq a<b$ the estimate
\begin{equation} \label{eq:jip}
L_{f_{k,i}}^{\ast}(b)-L_{f_{k,i}}^{\ast}(a)=
(b-a)\left(\frac{1}{3}+o(1)\right)+O(1), 
\qquad i\in\{1,2\}, \; k\to\infty,
\end{equation}
holds.
On the other hand, the set 
$\{ W_{k}E_{t},W_{k}E_{t-1},W_{k+1},TW_{k+1}\}$ is
linearly independent for any large $k$,
as they span the same space as $\mathscr{G}_{k}$.
It is not hard to verify that
\[
\max \{ \vert W_{k+1}(\zeta)\vert, \vert \zeta W_{k+1}(\zeta)\vert\}
< \min_{i=1,2} \vert U_{k,i}(\zeta)\vert,
\]
as soon as $H(U_{k,i})<H(W_{k+2})$, thus we have 
\begin{equation} \label{eq:wies}
L_{3,4}^{\ast}(b)\leq \max_{i=1,2} L_{f_{k,i}}^{\ast}(b),
\qquad q_{k}\leq b\leq q_{k+1}.
\end{equation}
Consider the polynomials
$W_{k}R_{1,k}$ and $W_{k}R_{2,k}$ with $R_{i,k}$
from \eqref{eq:infact}. Form the proof of Theorem~\ref{vorb},
$R_{1,k}, R_{2,k}$ can be chosen $E_{t}, E_{t-1}$.
Therefore,
by \eqref{eq:virtus}, for the parameter $q=q_{k}$ we conclude
\begin{equation} \label{eq:ferzig}
\max_{i=1,2} L_{f_{k,i}}^{\ast}(q_{k}) \leq  \varepsilon q_{k}.
\end{equation}
Finally
\eqref{eq:showmal} follows from \eqref{eq:jip} with $a=q_{k}$,
\eqref{eq:wies} and \eqref{eq:ferzig}. The relation \eqref{eq:balzzeit} is implied.

We have just shown in \eqref{eq:jori} and \eqref{eq:balzzeit}
that in the interval $I_{k}$, the
first two successive minima asymptotically decay with slope $-1/3$
and the third and fourth asymptotically increase with slope
$1/3$. Since at $q=q_{k}$ all $\vert L_{3,j}^{\ast}(q)\vert$
are of order $o(q)$ by \eqref{eq:virtus}, we have
$L_{3,i}^{\ast}(b_{k})= -L_{3,j}^{\ast}(b_{k})+o(b_{k})$ and thus
$\psi_{3,i}^{\ast}(B_{k})= -\psi_{3,j}^{\ast}(B_{k})+o(1)$ 
for $i\in\{1,2\}$
and $j\in\{3,4\}$, as $k\to\infty$. We conclude that
relation \eqref{eq:brauche} can be complemented to
\begin{align}
\psi_{3,j}^{\ast}(B_{k})&=\frac{1-\eta_{k}}{3(1+\eta_{k})}+o(1), \qquad\quad j\in\{1,2\}, \; k\to\infty,  \label{eq:ha}  \\
\psi_{3,j}^{\ast}(B_{k})&=-\frac{1-\eta_{k}}{3(1+\eta_{k})}+o(1), \qquad j\in\{3,4\}, \; k\to\infty.  \label{eq:hnb}
\end{align}
Let $J_{k}:=[b_{k},q_{k+1})$. 
We first show that $I_{k}$ and $J_{k}$ have roughly the same length,
that is $b_{k}$ should lie roughly 
in the middle between $q_{k}$ and $q_{k+1}$ and
$\lim_{k\to\infty} \vert I_{k}\vert/\vert J_{k}\vert=1$.
More precisely we establish $q_{k+1}-b_{k}=b_{k}-q_{k}+o(q_{k})$,
or equivalently
\begin{equation} \label{eq:weneed}
b_{k}=\frac{q_{k}+q_{k+1}}{2}+o(q_{k}).
\end{equation}
Recall combination of \eqref{eq:jogenau} and \eqref{eq:sieg} yields
\begin{equation} \label{eq:takeshi}
\frac{\log H(W_{k+1})}{\log H(W_{k})}=\eta_{k+1}+o(1), \qquad k\to\infty.
\end{equation}
Combination of \eqref{eq:brauchtum}, \eqref{eq:vorhang}
and \eqref{eq:takeshi} yields 
\begin{align*}
b_{k}&=\frac{3}{4}( \log H(W_{k+1})-\log\vert W_{k+1}(\zeta)\vert)=\\
&=\frac{3}{4}( \log H(W_{k+1})+(1+2\eta_{k+1}+o(1))\log H(W_{k+1}))=\\
&=\frac{3}{2}(1+\eta_{k+1}+o(1))\log H(W_{k+1}).
\end{align*}
This is indeed as the same value we obtain from 
combining \eqref{eq:qk}, \eqref{eq:takeshi} and 
the hypothesis \eqref{eq:weneed}, as the calculation
\[
\frac{q_{k}+q_{k+1}}{2}=\frac{3\log H(W_{k+1})+3\log H(W_{k+2})}{2}=\frac{3}{2}(1+\eta_{k+1})\log H(W_{k+1}),
\]
shows.
Thus indeed \eqref{eq:weneed} must hold 
and $\vert I_{k}\vert/\vert J_{k}\vert=1+o(1)$ as $k\to\infty$.

We now conclude that in the interval $J_{k}$, the functions 
$L_{3,1}^{\ast}, L_{3,2}^{\ast}$
have asymptotic slope $1/3$, whereas the functions 
$L_{3,3}^{\ast}, L_{3,4}^{\ast}$ have asymptotic slope $-1/3$,
until they all meet (asymptotically) at $q_{k+1}$. 
%
%
Recall we have established that $L_{3,3}^{\ast}$ and 
$L_{3,4}^{\ast}$ both rise in $I_{k}$ with asymptotic slope
$1/3$. If we let $b=b_{k}$ and $a=q_{k}$ in \eqref{eq:balzzeit}, 
since $\vert L_{3,j}^{\ast}(q)\vert$ are all small at $q=q_{k}$ 
in view of \eqref{eq:virtus} and $q_{k}\asymp b_{k}$, this means
\[
\frac{L_{3,j}^{\ast}(b_{k})-L_{3,j}^{\ast}(q_{k})}{b_{k}-q_{k}}=
\frac{L_{3,j}^{\ast}(b_{k})}{b_{k}-q_{k}}+o(1)= \frac{1}{3}+o(1),\qquad j\in\{3,4\},\; k\to\infty.
\]
In combination with the fact that $q_{k+1}-b_{k}$ and $b_{k}-q_{k}$
are roughly equal by \eqref{eq:weneed} and since
$\vert L_{3,j}^{\ast}(q)\vert$ are all small at $q=q_{k+1}$ again
by \eqref{eq:virtus} with index shift $k\to k+1$ and
$q_{k+1}\asymp b_{k}$, we infer
\begin{equation} \label{eq:dadada}
\frac{L_{3,j}^{\ast}(q_{k+1})-L_{3,j}^{\ast}(b_{k})}{q_{k+1}-b_{k}}=
-\frac{L_{3,j}^{\ast}(b_{k})}{q_{k+1}-b_{k}}+o(1)= -\frac{1}{3}+o(1),\qquad j\in\{3,4\},\; k\to\infty.
\end{equation}
On the other hand the functions $L_{3,j}^{\ast}$ have slope at least $-1/3$, so $L_{3,3}^{\ast}, L_{3,4}^{\ast}$ must each decay asymptotically with slope $-1/3+o(1)$ in $J_{k}$, that is 
for $b_{k}\leq a<b\leq q_{k+1}$ we must have
\begin{equation} \label{eq:hnu}
L_{3,j}^{\ast}(b)-L_{3,j}^{\ast}(a)= 
\left(-\frac{1}{3}+o(1)\right)(b-a)+O(1), \qquad j\in\{3,4\},\; k\to\infty.
\end{equation}
From \eqref{eq:lsumme} we further deduce that the sum 
$L_{3,1}^{\ast}+L_{3,2}^{\ast}$
must asymptotically increase with slope $2/3$ in $J_{k}$. 
By this again we mean that for $b_{k}\leq a<b\leq q_{k+1}$ we have
\begin{equation} \label{eq:a1a}
L_{3,1}^{\ast}(b)+L_{3,2}^{\ast}(b)-L_{3,1}^{\ast}(a)-L_{3,2}^{\ast}(a)= \left(\frac{2}{3}+o(1)\right)\cdot (b-a)+O(1), \qquad k\to\infty.
\end{equation}
Now from a very similar argument as for \eqref{eq:showmal}
we derive that both $L_{3,1}^{\ast}$ and $L_{3,2}^{\ast}$ 
increase with slope $1/3+o(1)$ as $k\to\infty$ in 
(any not too small subinterval of) $J_{k}$. 
On the one hand, since $L_{3,i}(b_{k})=-L_{3,j}(b_{k})$
for $i\in\{1,2\}, j\in\{3,4\}$ by \eqref{eq:ha}, \eqref{eq:hnb}, 
similarly to \eqref{eq:dadada} we calculate
\begin{equation} \label{eq:notice}
\frac{L_{3,j}^{\ast}(q_{k+1})-L_{3,j}^{\ast}(b_{k})}{q_{k+1}-b_{k}}=-\frac{L_{3,j}^{\ast}(b_{k})}{q_{k+1}-b_{k}}+o(1)= \frac{1}{3}+o(1),\qquad j\in\{1,2\},\; k\to\infty.
\end{equation}
This shows that on average $L_{3,1}^{\ast}$ and $L_{3,2}^{\ast}$
increase with slope $1/3$ in $J_{k}$. For the refined
local version (in any subinterval of $J_{k}$), notice that
\eqref{eq:notice}
implies that both $L_{3,1}^{\ast}, L_{3,2}^{\ast}$
increase on average by $1/3$ in $J_{k}$. However, by essentially
the same argument as for \eqref{eq:exakter},
\eqref{eq:showmal}, each of $L_{3,1}^{\ast}, L_{3,2}^{\ast}$ can increase with slope
at most $1/3$ in any not too small subinterval of $J_{k}$. 
Indeed, considering the
polynomials $W_{k+1}E_{t}, W_{k+1}E_{t-1}$ and 
the derived functions $f_{k+1,i}^{\ast}$ 
from \eqref{eq:k+1}
and \eqref{eq:k+2}, essentially Lemma~\ref{technisch} yields
the estimate
\begin{equation} \label{eq:1a1}
L_{3,1}^{\ast}(b)\leq L_{3,2}^{\ast}(b)\leq 
(b-b_{k})\left(\frac{1}{3}+o(1)\right)+O(1),\qquad
b>b_{k},\quad k\to\infty,
\end{equation}
similarly to \eqref{eq:showmal}.
By combination of \eqref{eq:a1a} and \eqref{eq:1a1}
for $b_{k}\leq a<b\leq q_{k+1}$ we indeed derive
\begin{equation} \label{eq:nhu}
L_{3,j}^{\ast}(b)-L_{3,j}^{\ast}(a)= 
\left(\frac{1}{3}+o(1)\right)(b-a)+O(1), 
\qquad j\in\{1,2\},\; k\to\infty,
\end{equation}
as claimed.

The claims \eqref{eq:hnu} with \eqref{eq:nhu} 
for $b=q_{k+1}, a=b_{k}$, can in view of \eqref{eq:virtus}
be stated as
\[
\lim_{k\to\infty} -\frac{L_{3,1}^{\ast}(b_{k})}{q_{k+1}-b_{k}}=
\lim_{k\to\infty} -\frac{L_{3,2}^{\ast}(b_{k})}{q_{k+1}-b_{k}}= 
\lim_{k\to\infty} \frac{L_{3,3}^{\ast}(b_{k})}{q_{k+1}-b_{k}}= 
\lim_{k\to\infty} \frac{L_{3,4}^{\ast}(b_{k})}{q_{k+1}-b_{k}}= 
 \frac{1}{3}.
\]

We eventually put the above findings together 
to determine the exponents.
Fix any interval $[q_{k},q_{k+1})=I_{k}\cup J_{k}$,
and let $q\in [q_{k},q_{k+1})$. 
From \eqref{eq:jori}, \eqref{eq:balzzeit}, \eqref{eq:hnu}, and
\eqref{eq:1a1}, no matter if  we have $q\in [q_{k},b_{k})$ 
or $q\in [b_{k},q_{k+1})$, as $k\to\infty$ we derive
\begin{align*}
L_{3,j}^{\ast}(q) &\geq 
L_{3,j}^{\ast}(b_{k})+
\left(\frac{1}{3}+o(1)\right)\vert q-b_{k}\vert+O(1)
\geq L_{3,j}^{\ast}(b_{k})+o(q), \qquad j\in\{1,2\},  \\
L_{3,j}^{\ast}(q)&\leq 
L_{3,j}^{\ast}(b_{k})-\left(\frac{1}{3}+o(1)\right)\vert q-b_{k}\vert+O(1)
\leq L_{3,j}^{\ast}(b_{k})+o(q), \qquad j\in\{3,4\}.
\end{align*}
It is easy to check with \eqref{eq:ha} and \eqref{eq:hnb} that
for $Q_{k}:= e^{q_{k}}$ and $M_{k}=[Q_{k},Q_{k+1})$, these 
results imply
\begin{align}
\psi_{3,j}^{\ast}(Q)&\geq \psi_{3,j}^{\ast}(B_{k})+o(1)
= \frac{1-\eta_{k}}{3(1+\eta_{k})}+o(1), 
\qquad \qquad\quad Q\in M_{k},\; j\in\{1,2\}, \label{eq:openup} \\
\psi_{3,j}^{\ast}(Q)&\leq \psi_{3,j}^{\ast}(B_{k})+o(1)
=-\frac{1-\eta_{k}}{3(1+\eta_{k})}+o(1), 
\qquad\qquad Q\in M_{k},\;j\in\{3,4\}, \label{eq:downup}
\end{align}
as $k\to\infty$. Observe that the right interval end of $J_{k}$ and
the left interval end of the successive $I_{k+1}$
both equal $q_{k+1}$, and 
hence $I_{1},J_{1},I_{2},J_{2},\ldots$ forms
a partition of $[q_{2},\infty)$.
Hence $\cup_{k\geq 1} M_{k}$ form a partition of $[Q_{2},\infty)$, 
and \eqref{eq:openup}, \eqref{eq:downup} yield for $j\in\{1,2\}$
\begin{equation*}
\underline{\psi}_{3,j}^{\ast}= \liminf_{Q\to\infty} 
\psi_{3,j}^{\ast}(Q)= \liminf_{k\to\infty} \inf_{Q\in M_{k}} 
\psi_{3,j}^{\ast}(Q)= \liminf_{k\to\infty} 
\psi_{3,j}^{\ast}(B_{k})= \liminf_{k\to\infty} \frac{1-\eta_{k}}{3(1+\eta_{k})},
\end{equation*}
and for $j\in\{3,4\}$ the estimates
\begin{equation*}
\overline{\psi}_{3,j}^{\ast}= \limsup_{Q\to\infty} 
\psi_{3,j}^{\ast}(Q)= \limsup_{k\to\infty} \sup_{Q\in M_{k}} 
\psi_{3,j}^{\ast}(Q)=
\limsup_{k\to\infty} 
\psi_{3,j}^{\ast}(B_{k})= \limsup_{k\to\infty}-\frac{1-\eta_{k}}{3(1+\eta_{k})}.
\end{equation*}
Thus since $\limsup_{k\to\infty} \eta_{k}=\sigma^{-1}$ by \eqref{eq:sigmarder}, we infer
\[
\underline{\psi}_{3,1}^{\ast}= \underline{\psi}_{3,2}^{\ast}=\frac{\sigma-1}{3(\sigma+1)}, \qquad
\overline{\psi}_{3,3}^{\ast}= \overline{\psi}_{3,4}^{\ast}=-\frac{\sigma-1}{3(\sigma+1)}.
\]
Application of \eqref{eq:jaja}, \eqref{eq:umrechnen} and \eqref{eq:umrechnen2} yields
\begin{equation} \label{eq:dazu}
w_{3}(\zeta)=w_{3,2}(\zeta)=1+\frac{2}{\sigma}, \qquad \lambda_{3}(\zeta)=\lambda_{3,2}(\zeta)=\frac{1}{1+2\sigma},
\end{equation}
containing the claims of the theorem. Furthermore, we point out the method shows that both functions
$\vert L_{3,1}^{\ast}(q)-L_{3,2}^{\ast}(q)\vert$ and 
$\vert L_{3,3}^{\ast}(q)-L_{3,4}^{\ast}(q)\vert$ 
differ at most by $o(q)$ as $q\to\infty$, and that
very similarly as for extremal numbers~\cite[Remark~4.2 and (50),(51)]{ichlondon}, we obtain 
\begin{align}
w_{3,3}(\zeta)=w_{3,4}(\zeta)=\widehat{w}_{3}(\zeta)=\widehat{w}_{3,2}(\zeta)=3,   \label{eq:cut} \\
\lambda_{3,3}(\zeta)=\lambda_{3,4}(\zeta)=\widehat{\lambda}_{3}(\zeta)
=\widehat{\lambda}_{3,2}(\zeta)=\frac{1}{3}, \nonumber
\end{align}
and
\begin{equation} \label{eq:whutli}
\widehat{w}_{3,3}(\zeta)=\widehat{w}_{3,4}(\zeta)=
1+2\sigma, \qquad \widehat{\lambda}_{3,3}(\zeta)
=\widehat{\lambda}_{3,4}(\zeta)=\frac{\sigma}{\sigma+2}.
\end{equation}
\end{proof}

We infer $w_{3}^{\ast}(\zeta)=1+2/\sigma$ for $\zeta=\zeta_{\varphi}$ from the monotonicity
of the sequence $(w_{n}^{\ast}(\zeta))_{n\geq 1}$, inequality \eqref{eq:sternug},
Theorem~\ref{buglauthm} and Theorem~\ref{rest} via
\[
1+\frac{2}{\sigma}=w_{2}^{\ast}(\zeta)\leq w_{3}^{\ast}(\zeta)\leq w_{3}(\zeta)=1+\frac{2}{\sigma}.
\]
Before we establish the last claim \eqref{eq:wstern} of Theorem~\ref{neu}, we turn towards Theorem~\ref{sternchen}. 
As in the case of extremal numbers, the description of the combined graph of the functions $L_{3,j}^{\ast}(q)$ 
from Theorem~\ref{neu} allows for deducing Theorem~\ref{sternchen}. The
proof is again similar to the one of~\cite[Theorem~2.2]{ichlondon}. However, application of results of Davenport and Schmidt~\cite{davsh}
and Bugeaud~\cite{bugbuch} lead almost directly to 
a proof
of \eqref{eq:otherhand}, compared to the rather tehcnical
and lengthy proof of the corresponding claim 
in~\cite[Theorem~2.2]{ichlondon}.
These results enable us
to establish the new claim \eqref{eq:pere} as well.


\begin{proof}[Proof of Theorem~\ref{sternchen}]
It follows from the proof of Theorem~\ref{rest} that for any 
large $q$,
the first two successive minima functions
of the linear form problem related 
to $\psi_{3,1}^{\ast}(q), \psi_{3,2}^{\ast}(q)$ are induced by
polynomial multiples $W_{k}E_{t}, W_{k}E_{t-1}$ (or 
$W_{k+1}, TW_{k+1}$, the argument below remains 
essentially the same)
of some polynomial $W_{k}$.
For each $k\geq 1$ they obviously
span the same space as $\{ W_{k},TW_{k}\}$.
Since the degree of $W_{k}$ is
two, there is no irreducible polynomial of degree exactly three which lies in this space. Thus the optimal exponent in \eqref{eq:trio}
cannot exceed $w_{3,3}(\zeta)$.
On the other hand it follows from the proof of Theorem~\ref{rest} 
that $\underline{\psi}_{3,3}^{\ast}=\underline{\psi}_{3,4}^{\ast}=0$, or equivalently $w_{3,3}(\zeta)=w_{3,4}(\zeta)=3$ by \eqref{eq:umrechnen2}, as already noticed in \eqref{eq:cut}.
Thus indeed \eqref{eq:trio} has only finitely many solutions 
in $Q\in{\mathbb{Z}[T]}$ 
an irreducible polynomial of degree precisely three.
From \eqref{eq:trio} we infer \eqref{eq:genaudrei1} by a standard argument, 
namely if $\alpha$ is a root 
of a polynomial $P$ of degree $n$ and close to some real number
$\zeta$, then
\begin{equation} \label{eq:standard}
\vert P(\zeta)\vert\ll_{\zeta,n}  H(P)\vert\zeta-\alpha\vert.
\end{equation}

Next we show \eqref{eq:trio2} and \eqref{eq:genaudrei2}.
From essentially the vector space argument from the proof of \eqref{eq:trio} we obtain similarly that 
$\widehat{w}_{3,3}(\zeta)+\epsilon$ is 
an upper bound for the exponent in \eqref{eq:trio2} for certain arbitrary large $X$. 
On the other hand, the proof of Theorem~\ref{rest} 
and \eqref{eq:debre} shows
$\widehat{w}_{3,3}(\zeta)=1+2\sigma$, as pointed out in \eqref{eq:whutli}. 
We conclude \eqref{eq:trio2}, and further deduce \eqref{eq:genaudrei2} very similarly
as \eqref{eq:genaudrei1} from \eqref{eq:trio}.

The claims \eqref{eq:otherhand}, \eqref{eq:pere} follow
from \eqref{eq:glm}, essentially using
an argument of Davenport and Schmidt~\cite{davsh}, and variants due to Bugeaud~\cite{bugbuch}.
Indeed~\cite[Lemma~1]{davsh} claims, in our notation, 
that for any real transcendental
$\zeta$, the estimate
\begin{equation} \label{eq:ref}
\vert \zeta-\alpha\vert \leq
H(\alpha)^{-1/\widehat{\lambda}_{n}(\zeta)-1+\epsilon}
\end{equation}
has infintely many solutions in real algebraic integers $\alpha$ of
degree at most $n+1$. For $n=3$,
since we have shown
$\widehat{\lambda}_{3}(\zeta)=1/3$ for Sturmian
continued fractions, the right estimate
of \eqref{eq:otherhand}
follows for algebraic integers of degree at most $4$ in place
of algebraic numbers of degree precisely three.
To make the transition to cubic algebraic numbers, we essentially 
proceed as in~\cite[Theorem~2.11]{bugbuch}, incorporating also the comments
below~\cite[Theorem~2.11]{bugbuch}. As pointed out the method
of~\cite[Theorem~2.11]{bugbuch} with $n=d-1$ leads to a new
proof of~\cite[Theorem~2.9]{bugbuch}, which states
that for a real algebraic number $\zeta$ of 
degree $d$, we have
$\vert \zeta-\alpha\vert\ll_{\zeta} H(\alpha)^{-d}$ for infinitely
many real algebraic numbers $\alpha$ of degree precisely $n=d-1$.
When we let $d=4, n=d-1=3$, it
becomes that
for an algebraic number $\zeta$ of degree four we have
$\vert \zeta-\alpha\vert\ll_{\zeta} H(\alpha)^{-4}$, the estimate in \eqref{eq:otherhand}, for infinitely
many real numbers $\alpha$ of degree precisely three. We have to
adapt the method slightly since we deal with transcendental $\zeta$.
We claim that the same argument 
yields that for any real transcendental $\zeta$ the estimate
$\vert \zeta-\alpha\vert\leq H(\alpha)^{-n-1+\epsilon}$
has infinitely many
real algebraic $\alpha$ of degree precisely $n$, as soon 
as $\widehat{w}_{n}(\zeta)=n$.
The only time where $\zeta$ being algebraic was required
in the proof of~\cite[Theorem~2.11]{bugbuch}
is at the start when the Liouville
inequality~\cite[Theorem~A.1]{bugbuch} is applied. Its
purpose is to guarantee that for certain large $X$ we have
\begin{equation} \label{eq:theend}
\vert P(\zeta)\vert \gg_{d,\zeta} H(P)^{-d+1}\geq X^{-d+1}
\end{equation}
for all non-zero integer polynomials $P$ of degree at 
most $n$ and $H(P)\leq X$. However, for transcendental $\zeta$, by the definition of $\widehat{w}_{n}(\zeta)$ and
$n=d-1$, the estimate \eqref{eq:theend} holds as soon as
$\widehat{w}_{n}(\zeta)=n$, up to the multiplicative constant
replaced by an $\epsilon$ in the exponent. When $n=3$,
note in Theorem~\ref{neu} we have indeed verified
the condition $\widehat{w}_{3}(\zeta)=3$ for Sturmian continued fractions $\zeta$.
From this point on we proceed
as in the proof of~\cite[Theorem~2.11]{bugbuch} to obtain
the right estimate of \eqref{eq:otherhand}.
The left estimate of \eqref{eq:otherhand} 
is further implied by \eqref{eq:standard}.

For \eqref{eq:pere} we kind of dualize the argument above,
in the sense of exchanging best and uniform approximation
accordingly. 
In fact in the proof of~\cite[Lemma~1]{davsh} deals with
the inequality
$\vert \zeta-\alpha\vert \leq H(\alpha)^{-w_{n,n+1}(\zeta)-1+\epsilon}$ instead of \eqref{eq:ref}, 
and the highlighted result 
\eqref{eq:ref} is deduced by Mahler's duality. The proof
can be readily altered to show the 
analogous uniform version, that is the
estimate 
$\vert \zeta-\alpha\vert \leq H(\alpha)^{-1}X^{-\widehat{w}_{n,n+1}(\zeta)+\epsilon}$ has a real algebraic 
integer solution
$\alpha$ of degree precisely $n$ for all large $X$. We apply it to $n=3$ again.
Since $\widehat{w}_{3,4}(\zeta)=\lambda_{3}(\zeta)^{-1}=2\sigma+1$
as shown in Theorem~\ref{neu},
the right claim in \eqref{eq:pere} holds for algebraic integers
$\alpha$ of degree at most $4$. Again the transition
to cubic algebraic numbers 
can be carried out as in~\cite[Theorem~2.11]{bugbuch}, and 
yields the right
estimate in \eqref{eq:pere}. The left is inferred from \eqref{eq:standard} again.
\end{proof}  


We now show the final claim \eqref{eq:wstern} of Theorem~\ref{neu}.
From Theorem~\ref{buglauthm} we see that 
$\widehat{w}_{3}^{\ast}(\zeta)\geq \widehat{w}_{2}^{\ast}(\zeta)=2+\sigma$
for any $\zeta=\zeta_{\varphi}$. We need to show
the reverse inequality. The key idea is to derive
from the proof of Theorem~\ref{rest} that
for parameters $X$ slightly smaller than $H(W_{k+1})$ for large $k$,
the best algebraic approximation $\alpha$ of degree three
or less to $\zeta$ is given
by the root $\alpha_{k}$ of the quadratic polynomial $W_{k}$
closer to $\zeta$.
Then the claim follows from Lemma~\ref{laubug}.
We only sketch some arguments derived from
the proof of Theorem~\ref{rest} to avoid repeating
cumbersome computations.

Let $\epsilon>0$ small.
From Theorem~\ref{rest} and its proof we know that
for large $k$ and certain $Z_{k}$ roughly of size $H(W_{k+1})$,
the inequality 
\begin{equation} \label{eq:doexist}
H(P)\leq Z_{k}, \qquad
\vert P(\zeta)\vert \leq Z_{k}^{-1-\frac{2}{\eta_{k}}-\epsilon}
\end{equation}
has no three linearly independent
solutions in at most cubic polynomials $P$. 
More precisely we certainly have $H(W_{k+1})^{1-\delta}\leq Z_{k}<H(W_{k+1})^{1+\delta}$ for arbitrarily 
small $\delta>0$ and $k\geq k_{0}(\delta)$. This follows
basically from the fact that the essential
local maxima of $\psi_{3,3}^{\ast}(Q)$ (or $L_{3,3}^{\ast}(q)$) are 
attained close to $Q=B_{k}$ (or $q=b_{k}$), see \eqref{eq:hnb}, and
at these points they catch up with the falling slope
$\log H(W_{k+1})-q$ of $L_{W_{k+1}}^{\ast}(q)$. We spare the details.
On the other hand, as noticed in the proof
of Theorem~\ref{sternchen}, the two linearly solutions 
to \eqref{eq:doexist} that do exist
span the same space 
as $\{ W_{k},TW_{k}\}$, which contains only multiples
of $W_{k}$ and thus no irreducible
cubic polynomial. Hence for any irreducible cubic integer polynomial 
of height at most $Z_{k}$ we have the reverse inequality
for $\vert P(\zeta)\vert$. Thus, choosing $\delta$ small enough
compared to $\epsilon$, \eqref{eq:standard} yields that
for large $k$ the estimate
\[
H(\beta)\leq Z_{k}, \qquad
\vert \zeta-\beta\vert \ll_{n,\zeta} H(\beta)^{-1}Z_{k}^{-1-\frac{2}{\eta_{k}}-\epsilon}
\]
has no solution in real cubic algebraic numbers $\beta$, with
the implied constant from \eqref{eq:standard}. 
On the other hand, since $Z_{k}$ is essentially of
size $H(W_{k+1})$ and we can decrease it just a little
to satisfy $Z_{k}<H(W_{k+1})$ if necessary, 
from Lemma~\ref{laubug} we know that for large $k$
\begin{equation} \label{eq:openstan}
H(\beta)\leq Z_{k}, \qquad \vert \zeta-\beta\vert 
\leq H(\beta)^{-1}Z_{k}^{-2-\frac{1}{\eta_{k}}-\epsilon}
\end{equation}
has no solution among real algebraic numbers $\beta$
of degree at most $2$. Combination and $1/\eta_{k}<1$
yields that \eqref{eq:openstan} has no solution in
algebraic numbers of degree at most three. Now
recall by \eqref{eq:sigmarder}, \eqref{eq:eta} and 
Lemma~\ref{laubug} we have
$\limsup_{k\to\infty} \eta_{k}= \sigma^{-1}$.
Thus we find
arbitrarily large values of $k$ for which
$\eta_{k}^{-1}$ is close to $\sigma$, more precisely
\begin{equation} \label{eq:finder}
\eta_{k}^{-1}=\sigma+o(1), \qquad k\to\infty.
\end{equation}
Choosing such a sequence of $k$ as in \eqref{eq:finder},
from \eqref{eq:openstan} and letting $\epsilon\to 0$, we see
that $\widehat{w}_{3}^{\ast}(\zeta)\leq 2+\sigma=\widehat{w}_{2}^{\ast}(\zeta)$. The proof is finished.

We conclude with Theorem~\ref{nvier}. The proof of \eqref{eq:az} will be rather simple 
and only relies on the knowledge of the value of $w_{2}(\zeta_{\varphi})$
from Theorem~\ref{buglauthm} and duality arguments. The refined upper bounds in
\eqref{eq:neddoijesus} and \eqref{eq:neddoije} 
will be derived from the recent result~\cite[Theorem~2.9]{ichsuccessive}.
For $n=2$ in the notation of~\cite{ichsuccessive}, 
its claim becomes
\begin{equation}  \label{eq:abschluss}
\lambda_{n}(\zeta)\leq \max\left\{ \frac{w_{2}(\zeta)}{\widehat{w}_{2}(\zeta)\widehat{w}_{n-1}(\zeta)}, \frac{1}{\widehat{w}_{2}(\zeta)}  \right\},
\qquad n\geq 2,
\end{equation}
for any transcendental real $\zeta$ that 
satisfies $w_{1}(\zeta)<2$.

\begin{proof}[Proof of Theorem~\ref{nvier}]
Since the sequence $(\lambda_{n}(\zeta))_{n\geq 1}$ is non-increasing,
the upper bound in \eqref{eq:az} follows from \eqref{eq:equ} and the hypothesis $n\geq 3$.
The lower bound $1/n$ comes from \eqref{eq:ldiri}. 
We have to prove the remaining lower bound.
From Theorem~\ref{buglauthm} we know that there exist $P$ of degree two and arbitrarily 
large height $H(P)$ such that $\vert P(\zeta)\vert \leq H(P)^{-1-2/\sigma+\epsilon}$. For any such $P$
the $n-1$ polynomials $P_{0}(T)=P, P_{1}(T)=TP(T),\ldots,P_{n-2}(T)=T^{n-2}P(T)$ have degree at most $n$ and share 
the same estimates since obviously $H(P_{0})=H(P_{1})=\cdots =H(P_{n-2})$ and 
$P_{0}(\zeta)\asymp_{\zeta} P_{1}(\zeta)\asymp_{\zeta}\cdots \asymp_{\zeta} P_{n-2}(\zeta)$. Hence
\[
w_{n,n-1}(\zeta)\geq 1+\frac{2}{\sigma}.
\]
With \eqref{eq:umrechnen2} we obtain 
\[
\underline{\psi}_{n,n-1}^{\ast}\leq \frac{(n-1)\sigma-2}{2n(\sigma+1)}.
\]
Hence, by application of \eqref{eq:duales} and \eqref{eq:jaja},
we infer
\[
\underline{\psi}_{n,1} = -\overline{\psi}_{n,n+1}^{\ast} 
\leq \frac{n-1}{2} \underline{\psi}_{n,n-1}^{\ast}\leq 
\frac{(n-1)^{2}\sigma-2(n-1)}{4n(\sigma+1)}.
\]
Inserting in \eqref{eq:umrechnen} we obtain the right expression in the maximum
as a lower bound for $\lambda_{n}$, and thus have established \eqref{eq:az}.

The identity \eqref{eq:gleidohar} is true for $n\in\{1,2\}$
since any $\zeta_{\varphi}$ has bounded partial quotients
and $\lambda_{2}(\zeta_{\varphi})=1$ as well by Theorem~\ref{buglauthm}. For $n\geq 3$,
the identity follows from \eqref{eq:az} if and only if $\sigma=0$. For the estimates \eqref{eq:neddoijesus} and \eqref{eq:neddoije}, recall \eqref{eq:abschluss} applies since
$w_{1}(\zeta_{\varphi})=1<2$. Moreover, if
$\sigma_{\varphi}>0$, then $w_{2}(\zeta)=1+2/\sigma_{\varphi}<\infty$. The left expression bound
in \eqref{eq:neddoijesus} follows by inserting
the corresponding values from Theorem~\ref{buglauthm} and 
$\widehat{w}_{n-1}(\zeta)\geq n-1$, and checking that
it the left expression in \eqref{eq:abschluss} is larger
for $n$ in the given range. 
The bound in \eqref{eq:neddoije} follows similarly
checking that the right expression is the larger one 
for larger $n$. 
The right expression bound in \eqref{eq:neddoijesus}
only reproduces \eqref{eq:az}, derived
from \eqref{eq:equ}. 
Finally, as shown in~\cite[Theorem~2.1]{ichjra}, the property
\eqref{eq:ichneu} holds for any real number
which is not a $U$-number in Mahler's classification
of real numbers,
and it follows from Adamczewski and Bugeaud~\cite{b} that 
any Sturmian continued fraction with $\sigma>0$ is not a $U$-number.
\end{proof}

\vspace{0.5cm}

The author warmly thanks the referee for 
pointing out inaccuracies and giving several other
valuable advices

\end{document}